\documentclass[11pt]{amsart}
\usepackage{amsmath,amssymb,latexsym,amsmath,amsthm,amscd}
\usepackage{setspace}
\usepackage[all]{xy} \xyoption{arc}
\usepackage[left=3cm,top=4cm,right=3cm,bottom = 4cm]{geometry}
\usepackage{graphicx}
\usepackage[usenames,dvipsnames]{color}
\usepackage{hyperref}

\usepackage{charter}

\theoremstyle{plain}
\newtheorem{thm}{Theorem}
\newtheorem{lem}[thm]{Lemma}
\newtheorem{prop}[thm]{Proposition}
\newtheorem{cor}[thm]{Corollary}

\theoremstyle{definition}
\newtheorem{dfn}[thm]{Definition}
\newtheorem{ex}[thm]{Example}

\theoremstyle{remark}
\newtheorem{rmk}[thm]{Remark}

\DeclareMathOperator{\End}{End}

\begin{document}

\title{On indecomposable non-simple $\mathbb{N}$-graded vertex algebras}
\author{Phichet Jitjankarn, and Gaywalee Yamskulna}
\address{School of Science, Walailak University, Nakhon Si Thammarat, Thailand}
\email{jitjankarn@gmail.com}
\address{Department of Mathematics\\ Illinois State University, Normal, IL, USA}
\email{gyamsku@ilstu.edu}

\maketitle
\begin{abstract}

\noindent In this paper, we study an impact of Leibniz algebras on the algebraic structure of $\mathbb{N}$-graded vertex algebras. We provide easy ways to characterize indecomposable non-simple $\mathbb{N}$-graded vertex algebras $\oplus_{n=0}^{\infty}V_{(n)}$ such that $\dim V_{(0)}\geq 2$. Also, we examine the algebraic structure of $\mathbb{N}$-graded vertex algebras $V=\oplus_{n=0}^{\infty}V_{(n)}$ such that $\dim~V_{(0)}\geq 2$ and $V_{(1)}$ is a (semi)simple Leibniz algebra that has $sl_2$ as its Levi factor. We show that under suitable conditions this type of vertex algebra is indecomposable but not simple. Along the way we classify vertex algebroids associated with (semi)simple Leibniz algebras that have $sl_2$ as their Levi factor.
\end{abstract}

\section{Introduction}

\noindent The aim of this paper is to explore criteria for $\mathbb{N}$-graded vertex algebras $V=\oplus_{n=0}^{\infty}V_{(n)}$ such that $\dim V_{(0)}\geq 2$ to be indecomposable non-simple vertex algebras, and to study influences of (semi)simple Leibniz algebras on the algebraic structure of this type of vertex algebras. For a $\mathbb{N}$-graded vertex algebra $V=\oplus_{n=0}^{\infty}V_{(n)}$ such that $\dim~V_{(0)}\geq 2$, $V_{(0)}$ is a unital commutative associative algebra, $V_{(1)}$ is a Leibniz algebra. The skew symmetry and the Jacobi identity for the vertex algebra give rise to several compatibility relations. The additional structures on $V_{(0)}\oplus V_{(1)}$ are further summarized in the notion of a vertex algebroid. It is well known that one can construct a $\mathbb{N}$-graded vertex algebra $V=\oplus_{n=0}^{\infty}V_{(n)}$ from any vertex $A$-algebroid such that $V_{(0)}=A$ and the vertex algebroid $V_{(1)}$ is isomorphic to the given one (cf. \cite{GMS}). There are other methods for constructing $\mathbb{N}$-graded vertex algebras. From any 1-truncated conformal algebra one can form a Lie algebra, generalizing the construction of affine Lie algebras. By using this Lie algebra and a result in \cite{DLM2} one can construct a $\mathbb{N}$-graded vertex algebra and its modules (cf. \cite{LiY}). Also, one can construct a $\mathbb{N}$-graded vertex operator algebra by shifting a Virasoro element. Beginning with a vertex operator algebra $(V,Y,{\bf 1},\omega)$, one can replace $\omega$ by another conformal vector $\omega_h:=\omega+h(-2){\bf 1}$ so that $V$ is a $\mathbb{N}$-graded vertex operator algebra with the same Fock space, vacuum vector and set of fields as $V$ (cf. \cite{DoM2}). 

\vspace{0.2cm}

\noindent In \cite{DoM}, Dong and Mason showed that a $\mathbb{N}$-graded vertex operator algebra $V$ is local if and only if $V_{(0)}$ is a local algebra. Moreover, indecomposibility of $V$ is equivalent to $V_{(0)}$ being a  local algebra (cf. \cite{DoM}). Note that in order to prove this statement one needs to have a Virasoro element. In \cite{MaY}, Mason and the second author of this paper studied simple, self-dual $\mathbb{N}$-graded vertex operator algebras that are $C_2$-cofinite, and  proved several results along the lines that the vertex operators in a Levi factor of the Leibniz algebra $V_{(1)}$ generate an affine Kac-Moody vertex operator subalgebra. Also, if $V$ arises as a shift of a self-dual vertex operator algebra of CFT-type then $V_{(0)}$ has a `de Rham structure' with many of the properties of the de Rham cohomology of a complex connected manifold equipped with Poincare duality.

\vspace{0.2cm}

\noindent For this paper, we search for easy ways to characterize indecomposable non-simple $\mathbb{N}$-graded vertex algebras. We establish the following result.

\begin{thm}\label{Main Theorem nonsimple} Let $V=\oplus_{n=0}^{\infty}V_{(n)}$ be a $\mathbb{N}$-graded vertex algebra that satisfies the following properties:

\vspace{0.2cm}

\noindent (a) $2\leq \dim V_{(0)}<\infty$, $1\leq dim V_{(1)}<\infty$, $V$ is generated by $V_{(0)}$ and $V_{(1)}$;

\vspace{0.2cm}

\noindent (b) $V_{(0)}$ is a local algebra.

\vspace{0.2cm}
\noindent Assume that one of the following statements hold.

\vspace{0.2cm} 

\noindent (i) An ideal generated by $rad\langle~,~\rangle$ is not zero;

\vspace{0.2cm}

\noindent (ii) $V_{(0)}$ is not a simple module for a Lie $V_{(0)}$-algebroid $V_{(1)}/Ann_{V_{(1)}}(V_{(0)})$;

\vspace{0.2cm}

\noindent (iii) $V_{(0)}$ is not a simple module for a Lie $V_{(0)}$-algebroid $V_{(1)}/{V_{(0)}}D(V_{(0)})$. 

\vspace{0.2cm}

\noindent Then $V$ is an indecomposable non-simple vertex algebra. 
\end{thm}

\noindent Furthermore, we investigate $\mathbb{N}$-graded vertex algebras $V=\oplus_{n=0}^{\infty}V_{(n)}$ such that $\dim~V_{(0)}\geq 2$ and $V_{(1)}$ is a (semi)simple Leibniz algebra that has $sl_2$ as its Levi factor. A number of theorems for Lie algebras were generalized to Leibniz algebras such as Lie's Theorem, Engel's Theorem, Cartan's criterium and Levi's Theorem (cf. \cite{Ba, Ba2, Ba3, O}). 
However, there are many results in Lie Theory that can not be generalized to the Theory of Leibniz algebras. In fact, for (semi)simple Leibniz algebras it is not true that a representation can be decomposed to a direct sum of irreducible ones (cf. \cite{ DMS, FM}). This result motivates us to study an impact of (semi)simple Leibniz algebras on the algebraic structure of $\mathbb{N}$-graded vertex algebras $V=\oplus_{n=0}^{\infty}V_{(n)}$ in this paper. We use (semi)simple Leibniz algebras that have $sl_2$ as their Levi factor to study indecomposability and non-simplicity properties of $\mathbb{N}$-graded vertex algebras. We examine the structure of $\mathbb{N}$-graded vertex algebras $V=\oplus_{n=0}^{\infty}V_{(n)}$ that is generated by $V_{(0)}$ and $V_{(1)}$ such that $\dim~V_{(0)}\geq 2$ and $V_{(1)}$ is a (semi)simple Leibniz algebra that has $sl_2$ as its Levi factor and establish the following result. 

\begin{thm}\label{MainTheorem} Let $V=\oplus_{n=0}^{\infty}V_{(n)}$ be a $\mathbb{N}$-graded vertex algebra that satisfies the following properties:

\vspace{0.2cm}

\noindent (a) $2\leq \dim V_{(0)}<\infty$, $1\leq dim V_{(1)}<\infty$, $V$ is generated by $V_{(0)}$ and $V_{(1)}$;

\vspace{0.2cm}

\noindent (b) $V_{(0)}$ is not a trivial module of a Leibniz algebra $V_{(1)}$, $u_0u\neq 0$ for some $u\in V_{(1)}$;

\vspace{0.2cm}

\noindent (c) the Levi factor of $V_{(1)}$ equals $Span\{e,f,h\}$, $e_0f=h$, $h_0e=2e$, $h_0f=-2f$ and $e_1f=k{\bf 1}$. Here, $k\in\mathbb{C}\backslash \{0\}$.

\vspace{0.2cm}

\noindent Assume that one of the following statements hold.

\vspace{0.2cm} 

\noindent (i) $V_{(1)}$ is a simple Leibniz algebra;

\vspace{0.2cm}

\noindent (ii) $V_{(1)}$ is a semisimple Leibniz algebra and $Ker(D)\cap V_{(0)}=\{a\in V_{(0)}~|~b_0a=0\text{ for all } b\in V_{(1)}\}$. Here, $D$ is a linear operator on $V$ such that $D(v)=v_{-2}{\bf 1}$ for $v\in V$.

\vspace{0.2cm}

\noindent Then $V$ is indecomposable but not a simple vertex algebra. 
\end{thm}

\vspace{0.2cm}

\noindent This paper is organized as follows: in Section 2, we review some necessary background on Leibniz algebras. In Section 3, we recall definitions of a 1-truncated conformal algebra, a vertex algebroid, and their basic properties. We examine relations among these algebraic structures. We use properties of Leibniz algebras to study the algebraic structure of vertex algebroids. Moreover, we classify vertex algebroids associated with (semi)simple Leibniz algebras that have $sl_2$ as their Levi factor. In Section 4, we study criteria for $\mathbb{N}$-graded vertex algebras to be indecomposable non-simple vertex algebras, and prove Theorem \ref{Main Theorem nonsimple}. In addition, we investigate the structure of $\mathbb{N}$-graded vertex algebras $V=\oplus_{n=0}^{\infty}V_{(n)}$ that is generated by $V_{(0)}$ and $V_{(1)}$ such that $\dim~V_{(0)}\geq 2$ and $V_{(1)}$ is a (semi)simple Leibniz algebra that has $sl_2$ as its Levi factor and prove Theorem \ref{MainTheorem}.


\section{Leibniz Algebras}

\noindent In this section, we provide necessary background on Leibniz algebras. We give a definition of left (respective, right) Leibniz algebras, define solvable left Leibniz algebras, state the analogue of Levi's Theorem, and discuss about simple and semisimple Leibniz algebras.

\begin{dfn}\cite{DMS, FM}\ \ 

\vspace{0.2cm}

\noindent (i) A {\em left Leibniz algebra} $\mathfrak{L}$ is a $\mathbb{C}$-vector space equipped with a bilinear map $[~,~]:\mathfrak{L}\times\mathfrak{L}\rightarrow\mathfrak{L}$ satisfying the Leibniz identity $$[a,[b,c]]=[[a,b],c]+[b,[a,c]]$$ for all $a,b,c\in\mathfrak{L}$.

\vspace{0.2cm}

\noindent (ii) A {\em right Leibniz algebra} $\mathbb{L}$ is a $\mathbb{C}$-vector space equipped with a bilinear map $\{~,~\}:\mathbb{L}\times \mathbb{L}\rightarrow\mathbb{L}$ satisfying the identity $$\{\{u,v\},w\}=\{\{u,w\},v\}+\{u,\{v,w\}\}$$ for all $u,v,w\in \mathbb{L}$. 
\end{dfn}

\begin{rmk}\ \  

\vspace{0.2cm}

\noindent (i) Let $(\mathfrak{L},[~,~])$ be a left Leibniz algebra. If we set $\{u,v\}:=-[v,u]$ for all $u,v\in\mathfrak{L}$, then $(\mathfrak{L},\{\cdot,\cdot\})$ is a right Leibniz algebra.

\vspace{0.2cm}

\noindent (ii) Let $(\mathbb{L}, \{~,~\})$ be a right Leibniz algebra. If we set $[a,b]:=-\{b,a\}$ for all $a,b\in\mathbb{L}$, then $(\mathbb{L},[~,~])$ is a left Leibniz algebra.
\end{rmk}

\begin{rmk} Let $\mathfrak{L}$ be a left Leibniz algebra. We have $[[a,a],b]=0$ for all $a,b\in\mathfrak{L}$ since $[[x,y],z]=[x,[y,z]]-[y,[x,z]]$ for all $x,y,z\in\mathfrak{L}$. 
\end{rmk}

\begin{ex} Every Lie algebra is a left Leibniz algebra and a right Leibniz algebra.
\end{ex}

\begin{ex} Let $G$ be a Lie algebra and let $M$ be a skew-symmetric $G$-module (i.e., $[m,g]=0$ for all $g\in G$, $m\in M$). Then the vector space $Q=G\oplus M$ equipped with the multiplication $[u+m,v+n]=[u,v]+u\cdot n$ is a left Leibniz algebra. Here, $u,v\in G$, $m,n\in M$. 
\end{ex}

\vspace{0.2cm}

\noindent In this paper, we only focus on left Leibniz algebras. Also, throughout this paper, a Leibniz algebra always refer to a left Leibniz algebra.

\begin{dfn}\cite{DMS} Let $\mathfrak{L}$ be a left Leibniz algebra over $\mathbb{C}$. Let $I$ be a subspace of $\mathfrak{L}$. $I$ is a {\em left} (respectively, {\em right}) {\em ideal} of $\mathfrak{L}$ if $[\mathfrak{L}, I]\subseteq I$ (respectively, $[I,\mathfrak{L}]\subseteq I$). $I$ is an {\em ideal} of $\mathfrak{L}$ if it is both a left and a right ideal. 
\end{dfn}

\begin{ex} We define $$Leib(\mathfrak{L})=Span\{~[u,u]~|~u\in\mathfrak{L}~\}.$$ Clearly, $Leib(\mathfrak{L})=Span\{[u,v]+[v,u]~|~u,v\in\mathfrak{L}\}$. Since  
$[a,[b,b]]=[[a,b],b]+[b,[a,b]]$, and $[[b,b],u]=0$ for all $a,b,u\in\mathfrak{L}$, we can conclude that $Leib(\mathfrak{L})$ is an ideal of $\mathfrak{L}$. Moreover, for $v,w\in Leib(\mathfrak{L})$, $[v,w]=0$. 
\end{ex}

\begin{dfn}\cite{DMS} Let $(\mathfrak{L}, [~,~])$ be a left Leibniz algebra. The series of ideals $$...\subseteq \mathfrak{L}^{(2)}\subseteq \mathfrak{L}^{(1)}\subseteq \mathfrak{L}$$ where $\mathfrak{L}^{(1)}=[\mathfrak{L},\mathfrak{L}]$, $\mathfrak{L}^{(i+1)}=[\mathfrak{L}^{(i)},\mathfrak{L}^{(i)}]$ is called the {\em derived series} of $\mathfrak{L}$.  A left Leibniz algebra $\mathfrak{L}$ is {\em solvable} if $\mathfrak{L}^{(m)}=0$ for some integer $m\geq 0$. As in the case of Lie algebras, any left Leibniz algebra $\mathfrak{L}$ contains a unique maximal solvable ideal $rad(\mathfrak{L})$ called the  the {\em radical} of $\mathfrak{L}$ which contains all solvable ideals.
\end{dfn}

\begin{ex} $Leib(\mathfrak{L})$ is a solvable ideal.
\end{ex}

\begin{dfn}\cite{DMS}\ \  

\vspace{0.2cm}

\noindent (i) A left Leibniz algebra $\mathfrak{L}$ is {\em simple}  if $[\mathfrak{L},\mathfrak{L}]\neq Leib(\mathfrak{L})$, and $\{0\}$, $Leib(\mathfrak{L})$, $\mathfrak{L}$ are the only ideals of $\mathfrak{L}$.

\vspace{0.2cm}

\noindent (ii) A left Leibniz algebra $\mathfrak{L}$ is said to be {\em semisimple} if $rad(\mathfrak{L})=Leib(\mathfrak{L})$. 
\end{dfn}

\begin{rmk}\cite{DMS} 

\vspace{0.2cm}

\noindent The Leibniz algebra $\mathfrak{L}$ is semisimple if and only if the Lie algebra $\mathfrak{L}/Leib(\mathfrak{L})$ is semisimple. However, if $\mathfrak{L}/Leib(\mathfrak{L})$ is a simple Lie algebra then $\mathfrak{L}$ is not necessarily a simple Leibniz algebra. Also, $\mathfrak{L}$ is not necessary a direct sum of simple Leibniz ideals when $\mathfrak{L}/Leib(\mathfrak{L})$ is a semisimple Lie algebra . 
\end{rmk}

\begin{ex} \cite{DMS}

\noindent For a positive integer $m$, we set $V_m$ to be an irreducible $sl_2$-module of dimension $m$. Then $sl_2\oplus V_m$ is a simple Leibniz algebra with $Leib(sl_2\oplus V_m)=V_m$. 

\vspace{0.2cm}

\noindent Next, we set $U=V_m\oplus V_n$. A vector space $\mathfrak{L}=sl_2\oplus U$ is a Leibniz algebra with $Leib(\mathfrak{L})=U$. Clearly, $U$, $V_m$, $V_n$ are all different ideals of $\mathfrak{L}$. Hence, $\mathfrak{L}$ is not a simple Leibniz algebra although $\mathfrak{L}/Leib(\mathfrak{L})$ is a simple Lie algebra. Furthermore, observe that $\mathfrak{L}$ can not be written as a direct sum of simple Leibniz ideals.\end{ex}

\begin{thm}\cite{Ba2, DMS} Let $\mathfrak{L}$ be a left Leibniz algebra. 

\vspace{0.2cm}

\noindent (i) There exists a subalgebra $S$ which is a semisimple Lie algebra of $\mathfrak{L}$ such that $\mathfrak{L}=S \dot{+} rad(\mathfrak{L})$. As in the case of a Lie algebra, we call $S$ a Levi subalgebra or a Levi factor of $\mathfrak{L}$.

\vspace{0.2cm}

\noindent (ii) If $\mathfrak{L}$ is a semisimple Leibniz algebra then $\mathfrak{L}=(S_1\oplus S_2\oplus...\oplus S_k)\dot{+}Leib(\mathfrak{L})$, where $S_j$ is a simple Lie algebra for all $1\leq j\leq k$. Moreover, $[\mathfrak{L},\mathfrak{L}]=\mathfrak{L}$. 

\vspace{0.2cm}

\noindent (iii) If $\mathfrak{L}$ is a simple Leibniz algebra, then there exists a simple Lie algebra $S$ such that $Leib(\mathfrak{L})$ is an irreducible module over $S$ and $\mathfrak{L}=S\dot{+}Leib(\mathfrak{L})$.
\end{thm}

\begin{dfn} Let $\mathfrak{L}$ be a left Leibniz algebra. A left $\mathfrak{L}$-module is a vector space $M$ equipped with a $\mathbb{C}$-bilinear map $\mathfrak{L}\times M\rightarrow M; (u,m)\mapsto u\cdot m$ such that $$([u,v])\cdot m=u\cdot (v\cdot m)-v\cdot(u\cdot m)$$ for all $u,v\in \mathfrak{L},m\in M$. 

\vspace{0.2cm}

\noindent The usual definitions of the notions of submodule, irreducibility, complete reducibility, homomorphism, isomorphism, etc., hold for left Leibniz modules.
\end{dfn}

\begin{rmk} $Leib(\mathfrak{L})$ acts as zero on $M$. 
\end{rmk}
\begin{ex}\ \ 

\vspace{0.2cm}

\noindent (i) Every left Leibniz algebra is a left module over itself via the Leibniz multiplication. 

\vspace{0.2cm}

\noindent (ii) $\mathbb{C}$ is a left $\mathfrak{L}$-module via $u\cdot w=0$ for every $u\in \mathfrak{L},w\in\mathbb{C}$. This module is called the trivial left module of $\mathfrak{L}$.

\vspace{0.2cm}

\noindent (iii) Let $\mathfrak{g}$ be a Lie algebra and let $M$ be a $\mathfrak{g}$-module. Hence, $\mathfrak{L}=\mathfrak{g}\oplus M$ is a Leibniz algebra. Let $N$ be a $\mathfrak{g}$-module. Then $N$ is in fact a left $\mathfrak{L}$-module via $(g+m)\cdot u=g\cdot u$ for every $g\in\mathfrak{g}, m\in M$, $u\in N$. 
\end{ex}


\section{1-Truncated Conformal Algebras and Vertex Algebroids}

\noindent First, we recall definitions of a 1-truncated conformal algebra, and a vertex algebroid. We review their basic properties. Also, we discuss relationships among these algebraic objects. 

\vspace{0.2cm}

\noindent Next, we use properties of Leibniz algebras to study ideals and modules of vertex algebroids. In particular, we classify vertex algebroids associated with (semi)simple Leibniz algebras that have $sl_2$ as their Levi factor. 

\subsection{1-Truncated Conformal Algebras}
\begin{dfn}\cite{GMS} A {\em 1-truncated conformal algebra} is a graded vector space $C=C_0\oplus C_1$ equipped with a linear map $\partial:C_0\rightarrow C_1$ and bilinear operations $(u,v)\mapsto u_iv$ for $i=0,1$ of degree $-i-1$ on $C=C_0\oplus C_1$ such that the following axioms hold:

\medskip

\noindent(Derivation) for $a\in C_0$, $u\in C_1$,
\begin{equation}
(\partial a)_0=0,\ \ (\partial a)_1=-a_0,\ \ \partial(u_0a)=u_0\partial a;
\end{equation}

\noindent(Commutativity) for $a\in C_0$, $u,v\in C_1$,
\begin{equation} 
u_0a=-a_0u,\ \ u_0v=-v_0u+\partial(u_1v),\ \ u_1v=v_1u;
\end{equation}

\noindent(Associativity) for $\alpha,\beta,\gamma\in C_0\oplus C_1$,
\begin{equation}
\alpha_0\beta_i\gamma=\beta_i\alpha_0\gamma+(\alpha_0\beta)_i\gamma.
\end{equation}
\end{dfn}

\begin{prop}\label{C0C1}\cite{GMS, LiY} Let $C=C_0\oplus C_1$ be a graded vector space equipped with a linear map $\partial$ from $C_0$ to $C_1$ and equipped with  bilinear maps $(u,v)\mapsto u_iv$ of degree $-i-1$ on $C=C_0\oplus C_1$ for $i=0,1$. Then $C$ is a 1-truncated conformal algebra if and only if

\vspace{0.2cm}

\noindent (i) The pair $(C_1,[~,~])$ carries the structure of a Leibniz algebra where $[u,v]=u_0v$ for $u,v\in C_1$, and the space $C_0$ is a $C_1$-module with $u\cdot a=u_0a$ for $u\in C_1$, $a\in C_0$.

\vspace{0.2cm}

\noindent (ii) The map $\partial$ is a $C_1$-module homomorphism, and the subspace $\partial C_0$ annihilates the $C_1$-module $C_0\oplus C_1$.

\vspace{0.2cm}

\noindent (iii) The bilinear map $\langle~,~\rangle$ from $C_1\otimes C_1$ to $C_0$ defined by $\langle u,v\rangle =u_1v$ for $u,v\in C_1$ is a $C_1$-module homomorphism. Furthermore, 
$$u_0a=-a_0u,~\langle \partial a,u\rangle=-a_0u,~[u,v]+[v,u]=\partial \langle u,v\rangle,~\langle u,v\rangle=\langle v,u\rangle$$
for $a\in C_0$, $u,v\in C_1$.
\end{prop}

\begin{prop} Let $(\mathfrak{g},\{~,~\})$ be a Lie algebra equipped with a symmetric invariant bilinear form $\langle~,~\rangle:\mathfrak{g}\otimes \mathfrak{g}\rightarrow\mathbb{C}$. Let $M$ and $A_M$ be $\mathfrak{g}$-modules equipped with a $\mathfrak{g}$-module isomorphism $\phi$ from $A_M$ to $M$. We set 
$$C_0=\mathbb{C}\oplus A_M,\text{ and }C_1=\mathfrak{g}\oplus M.$$ We define a Leibniz bracket on $C_1$ in the following way: for $g,g'\in \mathfrak{g},m,m'\in M$,
$$[g+m,g'+m']=\{g,g'\}+g\cdot m'.$$ Clearly, $C_0$ is an $C_1$-module under the following action: for $g\in\mathfrak{g},m\in M, \alpha\in\mathbb{C}, a\in A_M$,
\begin{equation}(g+m)\cdot (\alpha+a)=g\cdot a.\end{equation} 
Now, we let $\partial:C_0\rightarrow C_1$ be a linear map such that $\partial(\alpha)=0$ and $\partial(a)=\phi(a)$ for all $\alpha\in \mathbb{C}$, $a\in A_M$. Next, we define a bilinear map $(v,w)\mapsto v_iw$ of degree $-i-1$ on $C=C_0\oplus C_1$ in the following way: for $g,g'\in \mathfrak{g}$, $\lambda,\beta\in C_1$, $a,a'\in C_0$,
\begin{eqnarray*}
&&\lambda_0\beta=[\lambda,\beta], ~\lambda_0a=\lambda\cdot a,~a_0\lambda=-\lambda\cdot a,~a_0a'=0,\\
&&g_1g'=\langle g,g'\rangle,~g_1\partial(a)=g\cdot a,~\partial(a)_1g=g\cdot a, ~\partial(a)_1\partial(a')=0,\\
&&a_1a'=0,~a_1\lambda=0,~\lambda_1a=0.
\end{eqnarray*}
Then $C_0\oplus C_1$ is a 1-truncated conformal algebra.
\end{prop}
\begin{proof} Observe that $\partial(C_0)\subseteq M$ and $\partial(C_0)$ annihilates $C_0\oplus C_1$. Moreover, $\partial$ is a $C_1$-module homomorphism. For $\lambda,\beta\in C_1$, we have $\lambda_1\beta=\beta_1\lambda$, and $\lambda_0\beta=-\beta_0\lambda+\partial(\lambda_1\beta)$.
Since 
\begin{eqnarray*}
&&((u+\partial(q))\cdot (g+\partial(a)))_1 (g'+\partial(a'))+(g+\partial(a))_1 ((u+\partial(q))\cdot (g'+\partial(a')))\\
&&=([u,g]+\partial(u\cdot a))_1 (g'+\partial(a'))+(g+\partial(a))_1 ([u,g']+\partial(u\cdot a'))\\
&&=u\cdot (g\cdot a')+u\cdot (g'\cdot a)\\
&&=(u+\partial(q))\cdot (g+\partial(a))_1(g'+\partial(a'))
\end{eqnarray*}
for all $u,g,g'\in\mathfrak{g}$, $q,a,a'\in A$, we can conclude that $C_1\otimes C_1\rightarrow C_0: b\otimes b'\mapsto b_1b'$ is a $C_1$-module homomorphism. By Proposition \ref{C0C1}, $C_0\oplus C_1$ is a 1-truncated conformal algebra. 
\end{proof}

\begin{cor} Let $\mathfrak{g}$ be a Lie algebra over $\mathbb{C}$ equipped with a symmetric invariant bilinear form $\langle~,~\rangle$. Let $\partial$ be the zero map from $\mathbb{C}$ to $\mathfrak{g}$. Then $C=\mathbb{C}\oplus \mathfrak{g}$ is a 1-truncated conformal algebra.
\end{cor}

\subsection{Vertex Algebroids}
\begin{dfn}\cite{Br1,Br2, GMS} Let $(A,*)$ be a unital commutative associative algebra over $\mathbb{C}$ with the identity $1$. A {\em vertex $A$-algebroid} is a $\mathbb{C}$-vector space $\Gamma$ equipped with 
\begin{enumerate}
\item a $\mathbb{C}$-bilinear map $A\times \Gamma\rightarrow \Gamma, \ \ (a,v)\mapsto a\cdot v$ such that $1\cdot v=v$ (i.e. a nonassociative unital $A$-module),
\item a structure of a Leibniz $\mathbb{C}$-algebra $[~,~]:\Gamma\times \Gamma\rightarrow\Gamma$, 
\item a homomorphism of Leibniz $\mathbb{C}$-algebra $\pi:\Gamma\rightarrow Der(A)$,
\item a symmetric $\mathbb{C}$-bilinear pairing $\langle ~,~\rangle:\Gamma\otimes_{\mathbb{C}}\Gamma\rightarrow A$,
\item a $\mathbb{C}$-linear map $\partial :A\rightarrow \Gamma$ such that $\pi\circ \partial =0$ which satisfying the following conditions:
\begin{eqnarray*}
&&a\cdot (a'\cdot v)-(a*a')\cdot v=\pi(v)(a)\cdot \partial(a')+\pi(v)(a')\cdot \partial(a),\\
&&[u,a\cdot v]=\pi(u)(a)\cdot v+a\cdot [u,v],\\
&&[u,v]+[v,u]=\partial(\langle u,v\rangle),\\
&&\pi(a\cdot v)=a\pi(v),\\
&&\langle a\cdot u,v\rangle=a*\langle u,v\rangle-\pi(u)(\pi(v)(a)),\\
&&\pi(v)(\langle v_1,v_2\rangle)=\langle [v,v_1],v_2\rangle+\langle v_1,[v,v_2]\rangle,\\
&&\partial(a*a')=a\cdot \partial(a')+a'\cdot\partial(a),\\
&&[v,\partial(a)]=\partial(\pi(v)(a)),\\
&&\langle v,\partial(a)\rangle=\pi(v)(a)
\end{eqnarray*}
for $a,a'\in A$, $u,v,v_1,v_2\in\Gamma$.
\end{enumerate}
\end{dfn}

\begin{prop}\cite{LiY} Let $(A,*)$ be a unital commutative associative algebra and let $B$ be a module for $A$ as a nonassociative algebra . Then a vertex $A$-algebroid structure on $B$ exactly amounts to a 1-truncated conformal algebra structure on $C=A\oplus B$ with 
\begin{eqnarray*}
&&a_ia'=0,\\
&&u_0v=[u,v],~u_1v=\langle u,v\rangle,\\
&&u_0a=\pi(u)(a),~ a_0u=-u_0a
\end{eqnarray*} for $a,a'\in A$, $u,v\in B$, $i=0,1$ such that 
\begin{eqnarray*}
&&a\cdot(a'\cdot u)-(a*a')\cdot u=(u_0a)\cdot \partial a'+(u_0a')\cdot \partial a,\\
&&u_0(a\cdot v)-a\cdot (u_0v)=(u_0a)\cdot v,\\
&&u_0(a*a')=a*(u_0a')+(u_0a)*a',\\
&&a_0(a'\cdot v)=a'*(a_0v),\\
&&(a\cdot u)_1v=a*(u_1v)-u_0v_0a,\\
&&\partial(a*a')=a\cdot \partial(a')+a'\cdot \partial(a).
\end{eqnarray*}
\end{prop}

\begin{dfn} Let $I$ be a subspace of a vertex $A$-algebroid $B$. The vector space $I$ is called an {\em ideal } of the vertex $A$-algebroid $B$ if $I$ is a left ideal of the Leibniz algebra $B$ and $a\cdot u\in I$ for all $a\in A$, $u\in I$
\end{dfn}

\begin{ex} Let $(A,*)$ be a unital commutative associative algebra. Let $B$ be a vertex $A$-algebroid. We set $A\partial(A):=Span\{a\cdot\partial(a')~|~a,a'\in A\}$. Observe that \begin{eqnarray*}
b_0(a\cdot \partial(a'))&&=a\cdot (b_0\partial(a'))+(b_0a)\cdot \partial(a')=a\cdot \partial(b_0a')+(b_0a)\cdot \partial(a'),\text{ and }\\
a''\cdot(a\cdot\partial(a'))&&=(a''*a)\cdot \partial(a')
\end{eqnarray*} 
are in $A\partial(A)$ for all $b\in B, a,a'\in A$. Hence, $A\partial(A)$ is an ideal of the vertex $A$-algebroid $B$.
\end{ex}

\begin{rmk} $A\partial(A)$ is an abelian Lie algebra.
\end{rmk}

\vspace{0.2cm}

\noindent For the rest of this section, {\em we assume that 

\vspace{0.2cm}

\noindent (i) $(A,*)$ is a finite dimensional unital commutative associative algebra with the identity $\hat{1}$.

\vspace{0.2cm}

\noindent (ii) $B$ is a finite dimensional vertex $A$-algebroid.} 

\begin{prop}\label{radann} \ \ 

\vspace{0.2cm}

\noindent (i) We set $rad \langle ~,~\rangle:=\{u\in B~|~u_1b=0\text{ for all }b\in B\}$. Then $rad\langle~,~\rangle$ is an ideal of the left Leibniz algebra $B$. Also, $rad\langle~,~\rangle$ is a Lie algebra. 

\vspace{0.2cm}

\noindent (ii) We set $Ann_{B}(A):=\{b\in B~|~b_0a=0\text{ for all }a\in A\}$. Then $Ann_{B}(A)$ is an ideal of the left Leibniz algebra $B$. Moreover, $rad\langle~,~\rangle\subseteq Ann_{B}(A)$, and $Leib(B)\subseteq\partial(A)\subseteq Ann_{B}(A)$.

\vspace{0.2cm}

\noindent (iii) $Ann_B(A)$ is a module of the unital commutative associative algebra $A$. Moreover, $Ann_B(A)$ is an ideal of the vertex algebroid $B$. 

\vspace{0.2cm}

\noindent (iv) $rad\langle~,~\rangle$ is an $A$-submodule of $Ann_B(A)$. Moreover, $rad\langle~,~\rangle$ is an ideal of the vertex $A$-algebroid $B$. 
\end{prop}

\begin{proof} First, we will prove statement $(i)$. We recall that for $q,u,v\in B$, $$q_1(u_0v)=u_0q_1v-(u_0q)_1v.$$ 
Let $b\in B$, $w\in rad\langle~,~\rangle$. Since $t_1(b_0w)=b_0t_1w-(b_0t)_1w=0$ for all $t\in B$, we can conclude that $b_0w\in rad\langle~,~\rangle$ and $rad\langle ~,~\rangle$ is a left ideal of the Leibniz algebra $B$. Since $w_0b=-b_0w+\partial(w_1b)=-b_0w$, this implies that $w_0b\in rad\langle~,~\rangle$ and $rad\langle~,~\rangle$ is a right ideal of the Leibniz algebra $B$. Moreover, $rad\langle ~,~\rangle$ is a Lie algebra. 

\vspace{0.2cm}

\noindent Next, we will prove statement $(ii)$. Clearly, $\partial(A)\subseteq Ann_{B}(A)$. Let $u\in B$ and $b\in Ann_{B}(A)$. Since $(u_0b)_0a=u_0(b_0a)-b_0(u_0a)=0$ for all $a\in A$, we can conclude that $u_0b\in Ann_{B}(A)$ and $Ann_{B}(A)$ is a left ideal of the Leibniz algebra $B$. Since $b_0u=-u_0b+\partial(b_1u)$, we then have that $b_0u\in Ann_{B}(A)$ and $Ann_{B}(A)$ is a right ideal of the Leibniz algebra $B$. Therefore, $Ann_{B}(A)$ is an ideal of the Leibniz algebra $B$. Next, we will show that $rad\langle~,~\rangle\subseteq Ann_{B}(A)$. Let $v\in rad\langle~,~\rangle$. Since $$v_0a=-a_0v=(\partial(a))_1v=0$$ for all $a\in A$, it follows that $rad\langle~,~\rangle\subseteq Ann_{B}(A)$. 

\vspace{0.2cm}

\noindent Next, we will prove statement $(iii)$. Let $b\in Ann_{B}(A)$, $a,\alpha\in A$. Observe that 
$$(a\cdot b)_0a'=-a'_0(a\cdot b)=-a*(a'_0b)=a*(b_0a')=0\text{ for all }a'\in A,$$ and 
$$(a*\alpha)\cdot b=a\cdot(\alpha\cdot b)-(b_0a)\cdot \partial(\alpha)-(b_0\alpha)\cdot\partial(a)=a\cdot(\alpha\cdot b).$$
Hence, $Ann_{B}(A)$ is a module of the commutative associative algebra $A$. Furthermore, $Ann_B(A)$ is an ideal of the vertex $A$-algebroid $B$ since $Ann_B(A)$ is an ideal of the Leibniz algebra $B$ and $a\cdot u\in Ann_B(A)$.

\vspace{0.2cm}

\noindent To prove $(iv)$, we notice that $(a\cdot u)_1b=a*(u_1b)-u_0b_0a=0$ for all $a\in A$, $u\in rad\langle~,~\rangle$, $b\in B$. Consequently, $rad\langle~,~\rangle$ is an $A$-submodule of $Ann_B(A)$. By $(i)$, we can conclude that $rad\langle~,~\rangle$ is an ideal of the vertex $A$-algebroid $B$. 
\end{proof}

\begin{cor} We have $A\partial(A)\subseteq Ann_B(A)$.
\end{cor}
\begin{proof} This follows immediately from the fact that $$(a\cdot b)_0a'=-a'_0(a\cdot b)=-a*(a'_0b)=a*(b_0a')\text{ for all }a,a'\in A, b\in B.$$
\end{proof}

\begin{lem}\label{annba} If $B$ is a simple Leibniz algebra , $Leib(B)\neq \{0\}$ and $B\neq Ann_{B}(A)$ then 
$$Leib(B)=\partial(A)=Ann_{B}(A).$$
Moreover, $Leib(B)$ is an ideal of the vertex $A$-algebroid $B$ and $rad\langle~,~\rangle=\{0\}$. 
\end{lem}

\begin{proof} Assume that $B$ is simple and $Leib(B)\neq \{0\}$, $B\neq Ann_{B}(A)$. Clearly, $Leib(B)=\partial(A)=Ann_{B}(A)$.  By Proposition \ref{radann} $(iii)$, we can conclude that $Leib(B)$ is an ideal of the vertex $A$-algebroid $B$. Since $B$ is simple and $B\neq Ann_{B}(A)$ , it follows that $rad\langle~,~\rangle$ is either $\{0\}$ or $Leib(B)$. If $rad\langle~,~\rangle=Leib(B)$, then for every $a\in A$, $b\in B$, $$0=\partial(a)_1b=-a_0b=b_0a.$$ Consequently, $Ann_{B}(A)=B$ which is impossible. Therefore, $rad\langle ~,~\rangle=\{0\}$.
\end{proof}

\vspace{0.2cm}

\noindent We set 
$$Ker(\partial):=\{a\in A~|~\partial(a)=0\}.$$ Clearly, $Ker(\partial)\neq \emptyset$ since $\partial(\hat{1})=0$.
\begin{prop}\label{ker} \ \ 

\vspace{0.2cm}

\noindent (i) $Ker(\partial)$ is a subalgebra of $A$. 

\vspace{0.2cm}

\noindent (ii) $B$ acts trivially on $Ker(\partial)$.

\vspace{0.2cm}

\noindent (iii) $B$ is a $Ker(\partial)$-module when we consider $Ker(\partial)$ as an associative algebra. Moreover, $\partial(A)$ is a $Ker(\partial)$-module.

\vspace{0.2cm}

\noindent (iv) $Ker(\partial)$ contains every idempotent of $A$. 
\end{prop}
\begin{proof}  Let $a,a'\in Ker(\partial)$. Since $\partial(a*a')=a\cdot\partial(a')+a'\cdot\partial(a)=0$, it follows that  $a*a'\in Ker(\partial)$. Hence, $Ker(\partial)$ is a subalgebra of $A$ as desired. This proves $(i)$. Since $b_0a=-a_0b=\partial(a)_1b=0$ for all $a\in Ker(\partial),b\in B$, we can conclude that $B$ acts trivially on $Ker(\partial)$. This proves $(ii)$.

\vspace{0.2cm}

\noindent Next, we will prove statement $(iii)$. Since 
$$a\cdot(a'\cdot u)-(a*a')\cdot u=(u_0a)\cdot \partial(a')+(u_0a')\cdot\partial(a)=0\text{ for all }a,a'\in Ker(\partial),b\in B,$$
we can conclude that $B$ is a $Ker(\partial)$-module when we consider $Ker(\partial)$ as an associative algebra. Let $\alpha\in A$, $a\in Ker(\partial)$. Since $a\cdot\partial(\alpha)=\partial(a*\alpha)-\alpha\cdot\partial(a)=\partial(a*\alpha)\in \partial(A)$, it follows that $\partial(A)$ is a $Ker(\partial)$-module.

\vspace{0.2cm}

\noindent Lastly, we will show that $Ker(\partial)$ contains every idempotent of $A$. We will following the proof in Propositon 4.3 of \cite{DoM}. First, we observe that the idempotents and the involutorial units of $A$ span the same space. Therefore, to show that $Ker(\partial)$ contains every idempotent, it is enough to show that $Ker(\partial)$ contains involutorial units. Let $a\in A$ such that $a*a=\hat{1}$. Since 
$$0=a\cdot\partial(\hat{1})=a\cdot\partial(a*a)=a\cdot(2a\cdot\partial(a))=2(a*a)\cdot\partial(a)+4((\partial(a))_0a)\partial(a)=2\partial(a),$$ 
this implies that $\partial(a)=0$ and $Ker(\partial)$ contains every idempotent of $A$ as required.
\end{proof}

\begin{cor} If $Ker(\partial)=\mathbb{C}\hat{1}$, then $A$ is a local algebra. 
\end{cor}

\begin{thm} Assume that $rad\langle~,~\rangle=\{0\}$. Then $Ker(\partial)=\{a\in A~|~u_0a=0\text{ for all }u\in B\}$. In particular, if $B$ is a simple Leibniz algebra, $Leib(B)\neq \{0\}$ and $B\neq Ann_{B}(A)$, then $Ker(\partial)=\{a\in A~|~u_0a=0\text{ for all }u\in B\}$
\end{thm}
\begin{proof} We set $A_0=\{a\in A~|~u_0a=0\text{ for all }u\in B\}$. By Proposition \ref{ker}, in order to show that $Ker(\partial)=A_0$, we only need to prove that $A_0\subseteq Ker(\partial)$. Let $a\in A_0$. Observe that $\partial(a)_1u=-a_0u=u_0a=0$ for all $u\in B$. Since $rad\langle~,~\rangle=\{0\}$, we can conclude immediately that $\partial(a)=0$. Therefore, $a\in Ker(\partial)$ and $A_0\subseteq Ker(\partial)$.
\end{proof}

\vspace{0.2cm}

\noindent Now we study the vertex $A$-algebroids $B$ when $B$ are (semi)simple Leibniz algebras that have $sl_2$ as their Levi factor. Recall that if $B$ is a simple Leibniz algebra such that $Leib(B)\neq 0$, then there exists a simple Lie algebra $S$ such that $Leib(B)$ is an irreducible $S$-module and $B=S\dot{+}Leib({B})$. 

\begin{lem}\label{dimofA} Assume that $B$ is a simple Leibniz algebra such that 

\vspace{0.2cm}

\noindent (i) $Leib(B)\neq\{0\}$ and 

\vspace{0.2cm}

\noindent (ii) $S=Span\{e,f,h\}$, $e_0f=h$, $h_0e=2e$, $h_0f=-2f$, and $e_1f=k\hat{1}\in(\mathbb{C}\hat{1})\backslash\{0\}$. 

\vspace{0.2cm}

\noindent If $A$ is a direct sum of irreducible $sl_2$-modules $\mathbb{C}{\hat{1}}$ and $N$, then the dimension of $N$ is greater than 1. 
\end{lem}
\begin{proof} Notice that $Leib(B)=\partial(A)=\partial(\mathbb{C}\hat{1}\oplus N)=\partial(N)$. Let us assume that $\dim~N=1$. We set $N=\mathbb{C}a$. Hence, $N$ is a trivial module of $sl_2$. In fact, $N$ is a trivial $B$-module. Let $\{\hat{1},~a\}$ be a basis of $A$. Recall that there exist only two 2-dimensional unital commutative associative algebras. For the first algebra, we have $a*a=\hat{1}$ and for the second algebra, we have $a*a=0$. We set $a\cdot e=xe+yf+zh+\beta\partial(a)$. Here, $x,y,z,\beta\in \mathbb{C}$. Observe that for $u,v\in B$, we have $$u_0(a\cdot v)-a\cdot(u_0v)=(u_0a)\cdot v=0$$ because $N(=\mathbb{C}a)$ is a trivial $sl_2$-submodule of $A$. This implies that
\begin{eqnarray*}
&&0=h_0(a\cdot e)-a\cdot (h_0e)=2xe-2yf-2a\cdot e,\\
&&0=e_0(a\cdot e)-a\cdot (e_0e)=yh-2ze,\\
&&0=f_0(a\cdot e)-a\cdot (f_0e)=-xh+2zf+a\cdot h.
\end{eqnarray*}
Consequently, we have $y=z=0$, $a\cdot e=xe,\text{ and }a\cdot h=xh.$  

\vspace{0.2cm}

\noindent Observe that 
$a\cdot(a\cdot u)-(a*a)\cdot u=2(u_0a)\cdot \partial(a)=0\text{ for all }u\in B.$ 
If $a*a=\hat{1}$ then $e=(a*a)\cdot e=a\cdot(a\cdot e)=a\cdot (xe)=x^2 e$. Therefore, $x$ is either 1 or $-1$. 
If $a*a=0$ then $0=x^2e$. Hence, $x=0$

\vspace{0.2cm}

\noindent Recall that $(a\cdot u)_1v=a*(u_1v)-u_0v_0a=a*(u_1v)$ for all $u,v\in B$. Now, substituting $u$ (respectively, $v$) by $e$ (respectively, $f$), we have
\begin{eqnarray*}
&&a=\pm 1\text{ when }a*a=1,\\
&&a=0\text{ when }a*a=0.
\end{eqnarray*} 
This is a contradiction. Therefore, $\dim N>1$. 
\end{proof}

\vspace{0.2cm}

\noindent Recall that if $B$ is a semisimple Leibniz algebra such that $Leib(B)\neq\{0\}$, then there exist simple Lie algebras $S_1,..., S_t$ such that $B=(S_1\oplus....\oplus S_t)\dot{+} Leib(B)$. Moreover, $Leib(B)=\partial(A)$ since $Leib(B)=rad(B)$ and $Leib(B)\subseteq \partial(A)\subseteq rad(B)$. 

\vspace{0.2cm}

\noindent Now, let us focus on the case when $t=1$ and $S_1=Span\{e,f,h~|~e_0f=h,~h_0e=2e,~h_0f=-2f\}\cong sl_2$. Also, let us assume that $e_1f=k\hat{1}\in\mathbb{C}\hat{1}\backslash\{0\}$, and $A$ is a direct sum of irreducible $sl_2$-modules $\mathbb{C}\hat{1}$ and $N$. By following the proof in Lemma \ref{dimofA}, one can show that the dimension of $N$ is greater than 1. We summarize this discussion in the following Lemma.
\begin{lem} Assume that $B=S\dot{+}Leib(B)$ is a semisimple Leibniz algebra such that $Leib(B)\neq\{0\}$ and $S=Span\{e,f,h\}$ such that $e_0f=h$, $h_0e=2e$, $h_0f=-2f$, and $e_1f=k\hat{1}\in(\mathbb{C}\hat{1})\backslash\{0\}$. If $A$ is a direct sum of irreducible $sl_2$-modules $\mathbb{C}{\hat{1}}$ and $N$, then the dimension of $N$ is greater than 1. 
\end{lem}

\vspace{0.2cm}

\noindent{\em For the rest of this section we assume that $A$ is not a trivial module of the Leibniz algebra $B$.}

\begin{thm}\label{Bsimple} Let $B$ be a simple Leibniz algebra such that $Leib(B)\neq\{0\}$. 
Assume that its Levi factor $S=Span\{e,f,h\}$ such that $e_0f=h$, $h_0e=2e$, $h_0f=-2f$, and $e_1f=k\hat{1}\in(\mathbb{C}\hat{1})\backslash\{0\}$. Then 

\vspace{0.3cm}

\noindent (i) $e_1e=f_1f=e_1h=f_1h=0$, $k=1$, $h_1h=2\hat{1}$.

\vspace{0.2cm}

\noindent (ii) $Ker(\partial)=\mathbb{C}\hat{1}$

\vspace{0.2cm}

\noindent (iii) $Leib(B)$ is an irreducible $sl_2$-module of dimension 2. Moreover, as a $sl_2$-module, $A$ is a direct sum of a trivial module and an irreducible $sl_2$-module of dimension 2. 

\vspace{0.2cm}

\noindent (iv) $A$ is a local algebra. Let $A_{\neq 0}$ be an irreducible $sl_2$-submodule of $A$ that has dimension 2. Let $a_0$ be the highest weight vector of $A_{\neq 0}$ of weight 1 and let $a_1=f_0a_0$. Hence, the set $\{a_0,a_1\}$ forms a basis of $A_{\neq 0}$, the set $\{\hat{1},a_0,a_1\}$ is a basis of $A$, and the set $\{\partial(a_0),\partial(a_1)\}$ is a basis of $Leib(B)$. 

\vspace{0.2cm}

\noindent Relationships among $a_0,a_1,e,f,h,\partial(a_0),\partial(a_1)$ are desribed below:
\begin{eqnarray}
&&(\partial(a_0))_1e=0,~(\partial(a_0))_1f=a_1,~(\partial(a_0))_1h=a_0,\\
&&(\partial(a_1))_1e=a_0,~(\partial(a_1))_1f=0,~(\partial(a_1))_1h=-a_1,\\
&&a_0\cdot e=0,~a_0\cdot f=\partial(a_1),~a_0\cdot h=\partial(a_0),~a_0\cdot \partial(a_i)=0\text{ for }i\in\{0,1\},\\
&&a_1\cdot e=\partial(a_0),~a_1\cdot f=0,~a_1\cdot h=-\partial(a_1),~a_1\cdot \partial(a_i)=0\text{ for }i\in\{0,1\},\\
&&a_i*a_j=0\text{ for all }i,j\in\{0,1\}.
\end{eqnarray}
\end{thm}

\begin{proof} 
First, we notice that $B=S\dot{+}\partial (A)$ since $\{0\}\neq Leib(B)=\partial(A)$. Also, a $S$-module is a $B$-module because $Leib(B)$ acts trivially on module. In particular, an irreducible $S$-module is an irreducible $B$-module. 

\vspace{0.2cm}

\noindent Observe that as a $S$-module, we can write $A=A_0\oplus A_{\neq 0}$ where $A_0$ is a sum of trivial $sl_2$-submodules of $A$ and $A_{\neq 0}$ is a sum of irreducible $sl_2$-submodules of $A$ that are not trivial modules. 
Moreover, $A_0$ is a sum of trivial $B$-submodules of $A$ and $A_{\neq 0}$ is a sum of irreducible $B$-submodules of $A$ that are not trivial modules. Since $A_0=Ker(\partial)$, this implies that $Leib(B)=\partial(A)=\partial(A_{\neq 0})$. Since $Leib(B)$ is an irreducible $B$-module, and $\partial:A_{\neq 0}\rightarrow Leib(B)$ is a $B$-isomorphism, we can conclude that $A_{\neq 0}$ is an irreducible $B$-module. In particular, $A_{\neq 0}$ is an irreducible $sl_2$-module. 

\vspace{0.2cm}

\noindent Since $S$ is a Lie algebra, this implies that $e_1e,f_1f,e_1h, f_1h\in Ker(\partial)$. Since $Ker(\partial)=A_0$, 
and 
$$h_0(e_1e)=4e_1e,~h_0(f_1f)=-4f_1f,~h_0(e_1h)=2e_1h,~h_0(f_1h)=-2f_1h,$$ we can conclude that \begin{equation}e_1e=f_1f=e_1h=f_1h=0.\end{equation} 

\vspace{0.2cm}

\noindent For simplicity, we assume that $dim~ A_{\neq 0}=m+1$. Let $a_0$ be the highest weight vector of $A_{\neq 0}$ of weight $m$. We set $a_i=\frac{1}{i!}(f_0)^ia_0$ ($i\geq 0$). Hence, $\{a_0,....,a_{m}\}$ is a basis of $A_{\neq 0}$ and 
\begin{eqnarray*}
&&h_0a_i=(m-2i)a_i,~ f_0a_i=(i+1)a_{i+1},~ e_0a_i=(m-i+1)a_{i-1}.
\end{eqnarray*}
Since $e_0f_1h=f_1e_0h+(e_0f)_1h$, we then have that $0=f_1(-2e)+h_1h$. Therefore, $$h_1h=2f_1e=2k\hat{1}.$$ Using the fact that for $a\in A$, $u\in B$, $\partial(a)_1u=u_0a$, we obtain the following:
\begin{eqnarray*}
&&\partial(a_i)_1e=e_0a_i=(m-i+1)a_{i-1},\\
&&\partial(a_i)_1f=f_0a_i=(i+1)a_{i+1},\\
&&\partial(a_i)_1h=(m-2i)a_i.
\end{eqnarray*}

\vspace{0.2cm}

\noindent Now, we will show that $Ker(\partial)=\mathbb{C}\hat{1}$. Suppose that $Ker(\partial)\neq \mathbb{C}\hat{1}$. Then there exists $\alpha\in Ker(\partial)$ such that $\alpha\not\in\mathbb{C}\hat{1}$. Recall that for $a\in A$, $u,v\in B$, we have
$$(a\cdot u)_1v=a* (u_1v)-u_0v_0a.$$ We set 
$\alpha\cdot h=xe+yf+zh+\sum_{i=0}^m\beta_i\partial(a_i)$ where $x,y,z,\beta_i\in\mathbb{C}$. Since 
\begin{eqnarray*}
&&(\alpha\cdot h)_1h=2kz\hat{1}+\sum_{i=0}^m\beta_i\partial(a_i)_1h=2kz\hat{1}+\sum_{i=0}^m\beta_i(m-2i)a_i\text{ and },\\
&&\alpha*(h_1h)-h_0h_0\alpha=2k\alpha,
\end{eqnarray*}
we then have that $\sum_{i=0}^m\beta_i(m-2i)a_i=-2kz\hat{1}+2k\alpha\in Ker(\partial)\cap A_{\neq 0}=\{0\}$. This is a contradiction. Therefore, $Ker(\partial)=\mathbb{C}\hat{1}$. This proves {\it{(ii)}}.

\vspace{0.2cm}

\noindent Next, we will show that $\dim~Leib(B)=\dim~ A_{\neq 0}=2$. For $i\in\{0,1,...,m\}$, we set $$a_i\cdot e=x^ie+y^if+z^ih+\sum_j\beta^i_j\partial(a_j).$$ Here, $x^i,y^i,z^i, \beta^i_j\in\mathbb{C}$. Observe that 
\begin{eqnarray*}
&&(a_i\cdot e)_1f=x^ik\hat{1}+\beta^i_0a_1+\beta^i_1 2a_2+....+\beta^i_{m-1}ma_m\text{ and }\\
&&a_i*(e_1f)-e_0f_0a_i=(k-(i+1)(m-i))a_i.
\end{eqnarray*}

\vspace{0.2cm}

\noindent\underline{Case I: $i=0$} 

\vspace{0.2cm}

\noindent Since $(a_i\cdot e)_1f=a_i*(e_1f)-e_0f_0a_i$, it follows that $x^0=0$, $k=m$, $\beta^0_j=0$ for all $j\in\{0,...,m-1\}$. In particular, we have 
$a_0\cdot e=y^0f+z^0h+\beta^0_m\partial(a_m)$.
Since
\begin{eqnarray*}
&&(a_0\cdot e)_1h=2kz^0\hat{1}+\beta^0_m(m-2m)a_m\text{ and }\\
&&(a_0)*(e_1h)-e_0h_0(a_0)=-me_0(a_0)=0,
\end{eqnarray*}
we can conclude that $z^0=0$, $\beta^0_m=0$ and $a_0*e=y^0f$. Since
$$(a_0\cdot e)_1e=ky^0\hat{1},~(a_0)*(e_1e)-e_0e_0a_0=0$$
we have $y^0=0$. In conclusion, 
\begin{eqnarray}
&&a_0\cdot e=0,\text{ and }\\ 
&&k=m.\end{eqnarray} 

\vspace{0.2cm}

\noindent \underline{Case II: $i\neq 0$} 

\vspace{0.2cm}

\noindent We have 
$x^ik\hat{1}+\beta^i_0a_1+\beta^i_1 2a_2+....+\beta^i_{m-1}ma_m=i(i-m+1)a_i$. Hence, $x^i=0$, $\beta^i_{j-1}=(i-m+1)\delta_{j,i}$, and $$a_i\cdot e=y^if+z^ih+(i-m+1)\partial(a_{i-1})+\beta^i_m\partial(a_m).$$ 
Now, we will focus on the case when $i=m$. Using the fact that 
\begin{eqnarray*}
(a_m\cdot e)_1h&&=(y^mf+z^mh+\partial(a_{m-1})+\beta^m_m\partial(a_m))_1h\\
&&=2kz^m\hat{1}+(m-2(m-1))a_{m-1}+\beta^m_m(m-2m)a_m,\text{ and }\\
a_m*(e_1h)-e_0h_0a_m&&=-e_0(m-2m)a_m=-(m-m+1)(-m)a_{m-1}\\
&&=ma_{m-1}
\end{eqnarray*}
we obtain that $z^m=0$, $\beta^m_m=0$, $m=-m+2$. Hence, 
\begin{equation}\label{dimN}
m=1.\end{equation} 
This implies that $$\dim~Leib(B)=\dim A_{\neq 0}=2$$ and 
\begin{equation}e_1f=1\hat{1},~h_1h=2\hat{1}.\end{equation} This proves {\it{(i)}} and {\it{(iii)}}.
Since
$$(a_1\cdot e)_1e=(y^1f+\partial(a_0))_1e=y^1\hat{1},\text{ and }(a_1)*(e_1e)-e_0e_0a_1=0,$$
we can conclude that $y^1=0$ and 
\begin{equation}a_1\cdot e=\partial(a_0).\end{equation} 

\vspace{0.2cm}

\noindent Next, we will show that $a_0\cdot f=\partial(a_1)$, and $a_1\cdot f=0$. Recall that for $a\in A$, $u,v\in B$, we have $$u_0(a\cdot v)=a\cdot (u_0v)+(u_0a)\cdot v.$$ Hence, 
$$h_0(a_0\cdot f)=-a_0\cdot f,\text{ and }h_0(a_1\cdot f)=-3a_1\cdot f.$$ 
Since $h_0e=2e$, $h_0f=-2f$, $h_0h=0$, $h_0\partial(a_0)=\partial(a_0)$ and $h_0\partial(a_1)=-\partial(a_1)$, we can conclude immediately that \begin{equation}a_1\cdot f=0.\end{equation} 
Also, $a_0\cdot f=\beta \partial(a_1)$ for some scalar $\beta\in\mathbb{C}$. Since 
\begin{eqnarray*}
&&(a_0\cdot f)_1h=\beta(\partial(a_1))_1h=\beta h_0a_1=-\beta a_1\\
&&a_0*(f_1h)-f_0h_0a_0=-f_0a_0=-a_1
\end{eqnarray*}
we obtain that $\beta=1$ and \begin{equation}a_0\cdot f=\partial(a_1).\end{equation}

\vspace{0.2cm}

\noindent Next, we will show that $(a_0)\cdot h=\partial(a_0)$, $(a_1)\cdot h=-\partial(a_1)$. Observe that $$h_0(a_0\cdot h)=a_0\cdot h\text{ and }h_0(a_1\cdot h)=-a_1\cdot h.$$ These imply that 
$a_0\cdot h=\lambda \partial(a_0)$ and $a_1\cdot h=\gamma\partial(a_1)$ for some $\lambda,\gamma\in\mathbb{C}$. Since
\begin{eqnarray*}
&&(a_0\cdot h)_1f=\lambda(\partial(a_0))_1f=\lambda f_0a_0=\lambda a_1\text{ and }\\
&&a_0*(h_1f)-h_0f_0a_0=a_1
\end{eqnarray*}
we have $\lambda=1$ and 
\begin{equation}a_0\cdot h=\partial(a_0).\end{equation} 
Similarly, since 
\begin{eqnarray*}
&&(a_1\cdot h)_1e=\gamma (\partial(a_1))_1e=\gamma e_0a_1=\gamma a_0\text{ and }\\
&&a_1*(h_1e)-h_0e_0a_1=-h_0a_0=-a_0,
\end{eqnarray*}
we can conclude that $\gamma=-1$ and 
\begin{equation}a_1\cdot h=-\partial(a_1).\end{equation} 

\vspace{0.2cm}

\noindent Next, we will show that for $i,j\in\{0,1\}$, $a_i\cdot\partial(a_j)=0$. Observe that $$Leib(B)=\partial(A)=A\partial(A)=Ann_B(A).$$ 
Since 
\begin{eqnarray*}
&&h_0(a_0\cdot \partial(a_1))=0,~h_0(a_0\cdot \partial(a_0))=2a_0\cdot \partial(a_0),\\
&&h_0(a_1\cdot \partial(a_0))=0,~h_0(a_1\cdot \partial(a_1))=-2a_1\cdot \partial(a_1),
\end{eqnarray*}
and
\begin{equation} h_0\partial(a_0)=\partial(a_0),\text{ and }h_0\partial(a_1)=-\partial(a_1),
\end{equation} we can conclude that for $i,j\in\{0,1\}$
\begin{equation}
a_i\cdot \partial(a_j)=0.
\end{equation}

\vspace{0.2cm}

\noindent Now, we will show that for $i,j\in\{0,1\}$ $a_i*a_j=0$.
Since 
$$(a_0\cdot\partial(a_1))_1h=0\text{ and }a_0*(\partial(a_1)_1h)-\partial(a_1)_0h_0a_0=(-1)a_0*a_1,$$ 
these imply that
\begin{equation} a_0*a_1=0.\end{equation} 
Because
$$(a_0\cdot\partial(a_1))_1e=0\text{ and }a_0*(\partial(a_1)_1e)-\partial(a_1)_0e_0a_0=a_0*a_0$$
we have
\begin{equation} a_0*a_0=0.\end{equation}
Notice that 
$$(a_1\cdot\partial(a_0))_1f=0\text{ and }a_1*(\partial(a_0)_1f)-\partial(a_1)_0f_0a_1=a_1*a_1.$$
We then have \begin{equation}a_1*a_1=0.\end{equation}

\vspace{0.2cm}

\noindent Here is a summary of all relations that we have found:
\begin{eqnarray*}
&&(\partial(a_0))_1e=0,~(\partial(a_0))_1f=a_1,~(\partial(a_0))_1h=a_0,\\
&&(\partial(a_1))_1e=a_0,~(\partial(a_1))_1f=0,~(\partial(a_1))_1h=-a_1,\\
&&a_0\cdot e=0,~a_0\cdot f=\partial(a_1),~a_0\cdot h=\partial(a_0),~a_0\cdot \partial(a_i)=0\text{ for }i\in\{0,1\},\\
&&a_1\cdot e=\partial(a_0),~a_1\cdot f=0,~a_1\cdot h=-\partial(a_1),~a_1\cdot\partial(a_i)=0\text{ for }i\in\{0,1\},\\
&&a_i*a_j=0\text{ for all }i,j\in\{0,1\}.
\end{eqnarray*}
Since $Ker\partial=\mathbb{C}\hat{1}$, we can conclude that $A$ is a local algebra. This completes the proof for statement $(iv)$.
\end{proof}

\begin{rmk}\label{BsimpleVnotsimple} $B/Ann_{B}(A)\cong sl_2$ as a Lie algebra (because $Ann_{B}(A)=Leib(B)=\partial(A)$). The space $A_{\neq 0}=Span\{a_0,a_1\}$ is an irreducible $B/Ann_{B}(A)$-module. In addition, $A_{\neq 0}$ is a proper ideal of $A$.
\end{rmk}

\begin{thm}\label{Bsemisimple} Suppose that $B$ is a semisimple Leibniz algebra such that $Leib(B)\neq\{0\}$, and $Ker(\partial)=\{a\in A~|~u_0a=0\text{ for all }u\in B\}$. 

\vspace{0.2cm}

\noindent Assume that the Levi factor $S=Span\{e,f,h\}$ such that $e_0f=h, h_0e=2e, h_0f=-2f$ and $e_1f=k\hat{1}\in\mathbb{C}\hat{1}\backslash \{0\}$. We set $A=\mathbb{C}\hat{1}\oplus_{j=1}^l N^j$ where each $N^j$ is an irreducible $sl_2$-submodule of $A$. Then
 
\vspace{0.2cm}

\noindent (i) $e_1e=f_1f=e_1h=f_1h=0$, $k=1$, $h_1h=2\hat{1}$;

\vspace{0.2cm}

\noindent (ii) $Ker(\partial)=\mathbb{C}\hat{1}$;

\vspace{0.2cm}

\noindent (iii) For $j\in\{1,...,l\}$ $dim ~N^j=2$, and $dim Leib(B)=2l$; 

\vspace{0.2cm}

\noindent (iv) $A$ is a local algebra. For each $j$, we let $a_{j,0}$ be a highest weight vector of $N^j$ and $a_{j,1}=f_0(a_{j,0})$. Then $\{\hat{1}, a_{j,i}~|~j\in \{1,....,l\},~i\in\{0,1\}\}$ is a basis of $A$, and $\{\partial(a_{j,i})~|~j\in\{1,...,l\},~i\in\{0,1\}\}$ is a basis of $Leib(B)$.  

\vspace{0.2cm}

\noindent Relations among $a_{j,i},e,f,h,\partial(a_{j,i})$ are described below: 
\begin{eqnarray}
&&a_{j,i}*a_{j',i'}=0,\label{rel1}\\
&&a_{j,0}\cdot e=0,~a_{j,1}\cdot e=\partial(a_{j,0}),\\
&&a_{j,0}\cdot f=\partial(a_{j,1}),~a_{j,1}\cdot f=0,\\
&&a_{j,0}\cdot h=\partial(a_{j,0}),~a_{j,1}\cdot h=-\partial(a_{j,1}),\\
&&a_{j,i}\cdot \partial(a_{j',i'})=0,\label{rel5}\\
&&\partial(a_{j,i})_1e=e_0a_{j,i}=(2-i)a_{j,i-1},\\
&&\partial(a_{j,i})_1f=f_0a_{j,i}=(i+1)a_{j,i+1},\\
&&\partial(a_{j,i})_1h=h_0a_{j,i}=(1-2i)a_{j,i}.
\end{eqnarray}
\end{thm}

\begin{proof} Observe that $Leib(B)=\partial(A)=A\partial(A)=rad(B)$. We set $A=A_0\oplus A_{\neq 0}$. Here $A_0$ is a sum of trivial $sl_2$-submodules of $A$ and $A_{\neq 0}$ is a sum of irreducibles that are not trivial submodules of $A$. By the assumption $A_0=Ker(\partial)$, we have $Leib(B)=\partial(A_{\neq 0})$. Moreover, $A_{\neq 0}$ is isomorphic to $Leib(B)$ as $sl_2$-modules.

\vspace{0.2cm}

\noindent We set $A_{\neq 0}=\oplus_{j=1}^l N^j$. Here, $N^j$ is an irreducible $sl_2$-module of dimension $m_j+1$. Let $a_{j,0}$ be a highest weight vector of $N^j$ of weight $m_j$. We set $a_{j,i}=\frac{1}{i!}(f_0)^ia_{j,0}$. Hence, $\{a_{j,0},a_{j,1},...,a_{j, m_j}\}$ is a basis of $N^j$ and 
$$h_0(a_{j,i})=(m_j-2i)a_{j,i}, ~f_0(a_{j,i})=(i+1)a_{j,i+1}, ~e_0(a_{j,i})=(m_j-i+1)a_{j,i-1}.$$ 

\vspace{0.2cm}

\noindent Using the fact that $\{e,f,h,\partial(a_{j,i})~|~j\in\{1,...,l\}, i\in \{0,...,m_j\}\}$ is a basis of $B$ and following the proof in Theorem \ref{Bsimple} up to equation (\ref{dimN}), one can show that $Ker(\partial)=\mathbb{C}\hat{1}$, \begin{eqnarray*}
&&k=1, h_1h=2\hat{1},~e_1e=f_1f=e_1h=f_1h=0,\\
&&\partial(a_{j,i})_1e=e_0a_{j,i}=(m_j-i+1)a_{j,i-1},\\
&&\partial(a_{j,i})_1f=f_0a_{j,i}=(i+1)a_{j,i+1},\\
&&\partial(a_{j,i})_1h=h_0(a_{j,i})=(m_j-2i)a_{j,i},\\
&&m_j=1, dim A_{\neq 0}=dim Leib(B)=2l.
\end{eqnarray*} 
In fact, to obtain relations (\ref{rel1})-(\ref{rel5}), one only needs to modify the rest of the proof in Theorem \ref{Bsimple}.
\end{proof}

\vspace{0.2cm}

\begin{cor}\label{BsemisimpleVnotsimple} The following statements hold:

\noindent (i) $Ann_{B}(A)=\partial(A)=Leib(B)$;

\noindent (ii) $B/Ann_{B}(A)$ is isomorphic to $sl_2$ as a Lie algebra; 

\noindent (iii) each $N^j$ is an irreducible $B/Ann_{B}(A)$. Moreover, each $N^j$ is a proper ideal of $A$. 
\end{cor}

\begin{proof} We only need to prove $(i)$. The rest is clear. Observe that $Leib(B)=\partial(A)\subseteq Ann_{B}(A)$. Let $b=\alpha_e e+\alpha_f f+\alpha_h h+\sum_{j=1}^l\sum_{i=0}^1\beta_{j,i}\partial(a_{j,i})\in Ann_{B}(A)$. Since 
\begin{eqnarray*}
&&0=b_0a_{1,0}=(\alpha_e e+\alpha_f f+\alpha_h h+\sum_{j=1}^l\sum_{i=0}^1\beta_{j,i}\partial(a_{j,i}))_0a_{1,0}=\alpha_f a_{1,1}+\alpha_ha_{1,0},\text{ and}\\ 
&&0=b_0a_{1,1}=\alpha_e a_{1,0}-\alpha_h a_{1,1},
\end{eqnarray*}
it follows that $b=\sum_{j=1}^l\sum_{i=0}^1\beta_{j,i}\partial(a_{j,i})$, and $Leib(B)=\partial(A)= Ann_{B}(A)$. 
\end{proof}

\section{Vertex Algebras} 

\noindent First, we review necessary background on vertex algebras and their ideals. Next, we investigate criteria for $\mathbb{N}$-graded vertex algebras to be indecomposable and non-simple, and prove Theorem \ref{Main Theorem nonsimple}. In addition, we study the algebraic structure of $\mathbb{N}$-graded vertex algebras $V=\oplus_{n=0}^{\infty}V_{(n)}$ that is generated by $V_{(0)}$ and $V_{(1)}$ such that $\dim~V_{(0)}\geq 2$ and $V_{(1)}$ is a (semi)simple Leibniz algebra that has $sl_2$ as a Levi factor. Lastly, we prove Theorem \ref{MainTheorem}.

\vspace{0.2cm}

\noindent Now, we recall a definition of a vertex algebra (\cite{Bo1, FLM,  LLi}). A {\em vertex algebra} is a vector space $V$ equipped with a linear map 
\begin{eqnarray*}
Y:V&&\rightarrow \End(V)[[x,x^{-1}]]\\
v&&\mapsto Y(v,x)=\sum_{n\in\mathbb{Z}}v_nx^{-n-1}\text{ where } v_n\in\End(V)
\end{eqnarray*} 
and equipped with a distinguished vector ${\bf 1}$, the {\em vacuum vector}, such that for $u,v\in V$, 
\begin{eqnarray*}
&&u_nv=0\text{ for $n$ sufficiently large},\\
&&Y({\bf 1},x)=1,\\
&&Y(v,x){\bf 1}\in V[[x]],\text{ and }\lim_{x\rightarrow 0}Y(v,x){\bf 1}=v
\end{eqnarray*}
and such that 
\begin{eqnarray*}
&&x_0^{-1}\delta\left(\frac{x_1-x_2}{x_0}\right)Y(u,x_1)Y(v,x_2)-x_0^{-1}\delta\left(\frac{x_2-x_1}{-x_0}\right)Y(v,x_2)Y(u,x_1)\\
&&=x_2^{-1}\delta\left(\frac{x_1-x_0}{x_2}\right)Y(Y(u,x_0)v,x_2)
\end{eqnarray*}
the {\em Jacobi identity}.

\vspace{0.2cm}

\noindent From the Jacobi identity we have Borcherds' commutator formula and iterate formula:
\begin{eqnarray}
&&[u_m,v_n]=\sum_{i\geq 0}{m\choose i}(u_iv)_{m+n-i}\\
&&(u_mv)_nw=\sum_{i\geq 0}(-1)^i{m\choose i}(u_{m-i}v_{n+i}w-(-1)^mv_{m+n-i}u_iw)
\end{eqnarray}
for $u,v,w\in V$, $m,n\in \mathbb{Z}$. 

\vspace{0.2cm}

\noindent We define a linear operator $D$ on $V$ by $D(v)=v_{-2}{\bf 1}$ for $v\in V$. Then $$Y(v,x){\bf 1}=e^{xD}v\text{ for }v\in V,$$ and $$[D,v_n]=(Dv)_n=-nv_{n-1}\text{ for }v\in V, ~n\in\mathbb{Z}.$$ Moreover, for $u,v\in V$, we have $$Y(u,x)v =e^{xD}Y(v,-x)u ~(\text{skew symmetry}).$$

\vspace{0.2cm}

\noindent A vertex algebra $V$ equipped with a $\mathbb{Z}$-grading $V=\coprod_{n\in\mathbb{Z}}V_{(n)}$ is called a {\em $\mathbb{Z}$-graded vertex algebra} if ${\bf 1}\in V_{(0)}$ and if for $u\in V_{(k)}$ with $k\in\mathbb{Z}$ and for $m,n\in\mathbb{Z}$, $$u_mV_{(n)}\subseteq V_{(k+n-m-1)}.$$ 
A $\mathbb{N}$-graded vertex algebra is defined in the obvious way.

\begin{prop}\cite{GMS} 

\vspace{0.2cm}

\noindent If $V=\oplus_{n\in \mathbb{N}}V_{(n)}$ is a $\mathbb{N}$-graded vertex algebra then 

\vspace{0.2cm}

\noindent (i) $V_{(0)}$ is a commutative associative algebra with the identity ${\bf 1}$ and $V_{(1)}$ is a Leibniz algebra.

\vspace{0.2cm}

\noindent (ii) In fact, $V_{(0)}\oplus V_{(1)}$ is a 1-truncated conformal algebra.

\vspace{0.2cm}

\noindent (iii) Moreover, $V_{(1)}$ is a vertex $V_{(0)}$-algebroid.
\end{prop}

\vspace{0.2cm}

\noindent Recall that an ideal of the vertex algebra $V$ is a subspace $I$ such that $v_nw\in I$, and $w_nv\in I$ for all $v\in V$, $w\in I$ and $n\in\mathbb{Z}$. Hence, $D(w)=w_{-2}{\bf 1}\in I$. 

\vspace{0.2cm}

\noindent Under the condition that $D(I)\subseteq I$, the left ideal condition $v_nw\in I$ for all $v\in V,w\in I, n\in\mathbb{Z}$ is equivalent to the right ideal condition $w_mv\in I$ for all $v\in V, w\in I, m\in\mathbb{Z}$. 

\vspace{0.2cm}

\noindent Now, we will study indecomposibility property of $\mathbb{N}$-graded vertex algebras.

\begin{prop}\label{Vindecomposable} Let $V=\oplus_{n=0}^{\infty}V_{(n)}$ be a $\mathbb{N}$-graded vertex algebra such that $V_{(0)}$ is a finite dimensional commutative associative algebra and $\dim V_{(0)}\geq 2$. If $V_{(0)}$ is a local algebra then $V$ is indecomposable.
\end{prop}
\begin{proof} Assume that $V$ is decomposable. Then there exist nonzero proper ideals $U$ and $W$ such that $V=U\oplus W$. For $n\in\mathbb{N}$, we set $U_{(n)}=V_{(n)}\cap U$ and $W_{(n)}=V_{(n)}\cap W$. If $W_{(0)}=\{0\}$ then ${\bf 1}\in U_{(0)}$ and $U=V$ which is impossible. Hence $W_{(0)}\neq \{0\}$. Similarly, $U_{(0)}\neq \{0\}$. Becuause ${\bf 1}\in V_{(0)}=U_{(0)}\oplus W_{(0)}$,  there exists $u\in U_{(0)}$ and $w\in W_{(0)}$ such that ${\bf 1}=u+w$. Observe that $U_{(0)}$ and $W_{(0)}$ are proper ideals of commutative associative algebra $V_{(0)}$. Also, $u$ and $w$ are not units in $V_{(0)}$. Since $V_{(0)}$ is a local algebra and $u$ is not a unit, these imply that $w=1-u$ is a unit which is a contradiction. Therefore, $V$ is indecomposable. 
\end{proof}

\begin{rmk} It was shown in Theorem 2 of \cite{DoM} that if $V$ is a vertex operator algebra such that $V_{(n)}=0$ for $n\leq -1$, then $V$ is indecomposable if and only if $V_{(0)}$ is a  commutative associative local algebra with respect to the product $a*b=\frac{1}{2}(a_{-1}b+b_{-1}a)$. In particular, when $V$ is an $\mathbb{N}$-graded vertex operator algebra, $V$ is indecomposable if and only if $V_{(0)}$ is a  commutative associative local algebra.
\end{rmk}

\vspace{0.2cm}

\noindent Next Proposition is one of the key ingredients for proving Theorem \ref{MainTheorem}. In addition, it demonstrates an influence of (semi)simple Leibniz algebra on the algebraic structure of the $\mathbb{N}$-graded vertex algebras. 
\begin{prop}\label{Indecomposable} Let $V=\oplus_{n=0}^{\infty}V_{(n)}$ be a $\mathbb{N}$-graded vertex algebra that satisfies the following properties

\vspace{0.2cm}

\noindent (a) $2\leq \dim~ V_{(0)}<\infty$, $1\leq \dim~ V_{(1)}<\infty$;

\vspace{0.2cm}

\noindent (b) $V_{(0)}$ is not a trivial module for a Leibniz algebra $V_{(1)}$, $u_{0}u\neq 0$ for some $u\in V_{(1)}$;

\vspace{0.2cm}

\noindent (c) the Levi factor of $V_{(1)}$ equals $Span\{e,f,h\}$, $e_0f=h$, $h_0e=2e$, $h_0f=-2f$, and $e_1f=k\hat{1}$. Here, $k\in\mathbb{C}\backslash\{0\}$. 

\vspace{0.2cm}

\noindent Assume that one of the following statements hold.

\vspace{0.2cm}

\noindent (I) $V_{(1)}$ is a simple Leibniz algebra;

\vspace{0.3cm}

\noindent (II) $V_{(1)}$ is a semisimple Leibniz algebra and $Ker(D)\cap V_{(0)}=\{a\in V_{(0)}~|~b_0a=0\text{ for all } b\in V_{(1)}\}$. 

\vspace{0.2cm}

\noindent Then $V$ is indecomposable. 
\end{prop}

\begin{proof} This follows immediately from Theorem \ref{Bsimple}, Theorem \ref{Bsemisimple} and Theorem \ref{Vindecomposable}.
\end{proof}
\vspace{0.2cm}

\vspace{0.2cm}

\noindent Next, we study criteria for $\mathbb{N}$-graded vertex algebras to be non-simple. For a subset $S$ of the vertex algebra $V$, we define $\langle S\rangle$ to be the smallest vertex subalgebra containing $S$ and we call $\langle S\rangle $ the {\em vertex subalgebra generated by $S$}. It was shown in Proposition 3.9.3 of \cite{LLi} that 
$$\langle S\rangle=span\{ u^{(1)}_{n_1}...u^{(r)}_{n_r}{\bf 1}~|~r\in\mathbb{N},~u^{(1)},...,u^{(r)}\in S,~n_1,...,n_r\in\mathbb{Z}\}.$$

\vspace{0.2cm}

\vspace{0.2cm}

\noindent Let $V=\oplus_{n=0}^{\infty}V_{(n)}$ be a $\mathbb{N}$-graded vertex algebra such that $\dim V_{(0)}\geq 2$. For the rest of this section we assume that 

\vspace{0.2cm}

\noindent (i) $V_{(0)}$ and $V_{(1)}$ are finite dimensional;

\noindent (ii) $V$ is generated by $V_{(0)}$ and $V_{(1)}$. 

\vspace{0.2cm}

\noindent Then $V$ is spanned by $u^1_{n_1}....u^r_{n_r}{\bf 1}$ for all $u^i\in V_{(0)}\cup V_{(1)}$, $n_i\in \mathbb{Z}$, $r\geq 0$. Recall that for $a,a'\in V_{(0)}$, $b,b'\in V_{(1)}$, $m,n\in\mathbb{Z}$, we have
\begin{eqnarray}
\label{relab}&&[a_m,a'_n]=0,~[a_m,b_n]=(a_0b)_{m+n},~[b_n,a_m]=(b_0a)_{m+n},\\
\label{relbb}&&[b_m,b'_n]=(b_0b')_{m+n}+m(b_1b')_{m+n-1}.
\end{eqnarray}
By equation (\ref{relab}), $V$ is spanned by 
\begin{equation}\label{ba}v^1_{n_1}....v^s_{n_s}a^1_{m_1}...a^t_{m_t}{\bf 1}
\end{equation} for all $v^i\in V_{(1)}$, $a^j\in V_{(0)}$, $n_i,m_i\in\mathbb{Z}$, $s,t\geq 0$. 

\vspace{0.2cm} 

\noindent By using commutativity formulas (\ref{relab}), (\ref{relbb}), and the fact that $a_{-n-1}=\frac{1}{n}(D(a))_{-n}$ for all $n\geq 1$, we then have the following Proposition. 
\begin{prop}\label{spanV}  \  \  

\begin{enumerate}
\item Let $W$ be a subspace of $V$ that is spanned by $a^1_{m_1}...a^t_{m_t}{\bf 1}$ for all $a^j\in V_{(0)}$, $m_j\in\mathbb{Z}$, $t\geq 0$. The subspace $W$ is invariant under the action of $b_n$ for all $b\in V_{(1)}$, $n\geq 0$.
\item $V$ is spanned by $v^1_{-n_1}....v^s_{-n_s}a$ for all $v^i\in V_{(1)}$, $a\in V_{(0)}$, $n_i\geq 1$, $s\geq 0$.
\end{enumerate}
\end{prop}

\vspace{0.2cm}

\noindent For a subset $S$ of $V$, we denote by $(S)$ the smallest ideal of $V$ containing $S$. By Corollary 4.5.10 of \cite{LLi}, we have
$$(S)=Span\{v_nD^i(u)~|~v\in V,n\in\mathbb{Z}, i\geq 0, u\in S\}.$$

\vspace{0.2cm}

\noindent Let $\langle~,~\rangle:V_{(1)}\times V_{(1)}\rightarrow V_{(0)}$ be a symmetric bilinear map defined by $\langle u,v\rangle=u_1v$ for $u,v\in V_{(1)}$. By Proposition \ref{radann}, $rad\langle~,~\rangle$ is an ideal of the vertex $V_{(0)}$-algebroid $V_{(1)}$. Also, $rad\langle~,~\rangle$ is a Lie algebra and 
$$rad\langle~,~\rangle\subseteq Ann_{V_{(1)}}(V_{(0)})=\{b\in V_{(1)}~|~b_0a=0\text{ for all }a\in V_{(0)}\}.$$ 
Let $(rad\langle ~,~\rangle )$ be the smallest ideal of $V$ containing $rad\langle~,~\rangle$. So, we have
\begin{eqnarray*}
(rad\langle~,~\rangle)&&=Span\{v_nD^i(u)~|~v\in V,~n\in\mathbb{Z},~i\geq 0,~u\in rad\langle~,~\rangle\}.\\
\end{eqnarray*}
\begin{lem} 
$(rad\langle~,~\rangle)=Span\{b_nv~|~b\in rad\langle~,~\rangle, n\in\mathbb{Z}, v\in V\}.$
\end{lem}
\begin{proof} We set $K=Span\{b_nv~|~b\in rad\langle~,~\rangle, n\in\mathbb{Z}, v\in V\}$. Since $(rad\langle~,~\rangle)$ is an ideal and $rad\langle~,~\rangle\subseteq (rad\langle~,~\rangle)$, we can conclude that $K\subseteq (rad\langle~,~\rangle)$. In addition, $D(K)\subseteq K$ because $D(b_nv)=b_nD(v)-nb_{n-1}v\in K$ for all $b\in rad\langle~,~\rangle, v\in V, n\in\mathbb{Z}$.

\vspace{0.2cm}

\noindent Now, we will show that $K$ is an ideal of $V$. Let $a\in V_{(0)}$, $u\in V_{(1)}$, $b\in rad\langle~,~\rangle$, and let $m,n\in\mathbb{Z}$. Since $rad\langle~,~\rangle\subseteq Ann_{V_{(1)}}(V_{(0)})$, it follows that $$a_mb_nv=b_na_mv+(a_0b)_{m+n}v=b_na_mv-(b_0a)_{m+n}v=b_na_mv\in K.$$ 
Also, $$u_mb_nv=b_nu_mv+(u_0b)_{m+n}v+m(u_1b)_nv=b_nu_mv+(u_0b)_{m+n}v\in K$$ because $rad\langle~,~\rangle$ is an ideal of $V_{(1)}$. Since $V$ is generated by $V_{(0)}$ and $V_{(1)}$ we can conclude that $w_tb_nv\in K$ for all $w\in V$, $t\in\mathbb{Z}$ and $K$ is an ideal of $V$. Since $rad\langle~,~\rangle \subseteq K$, $K\subseteq (rad\langle~,~\rangle)$ and $(rad\langle~,~\rangle)$ is the smallest ideal that contains $rad\langle~,~\rangle$, we can conclude that $K=(rad\langle~,~\rangle)$.
\end{proof}

\noindent Next Theorem is an important theorem that will be used for proving Theorem \ref{Main Theorem nonsimple}.
\begin{thm}\label{rad} The intersection of $(rad\langle~,~\rangle)$ and $V_{(0)}$ equals $\{0\}$. If $V$ is simple then $(rad\langle~,~\rangle)=\{0\}$. Consequently, $V$ is non-simple when $(rad\langle~,~\rangle)\neq\{0\}$.
\end{thm}
\begin{proof}  We say that $v$ has {\em length} $t$ if $v$ is of the form $s^1_{-m_{1}}....s^t_{-m_{t}}q$ where $s^{i}\in V_{(1)}, q^i\in V_{(0)}, m_{i_j}>0$. We will use an induction on this type of length to show that for $u\in rad\langle~,~\rangle$, and for $v\in V$ such that $v=s^1_{-m_{1}}....s^t_{-m_{t}}q\in V_{(n)}$ where $s^{i}\in V_{(1)}, q\in V_{(0)}, m_{i_j}>0$, we have
\begin{equation}\label{induction}
u_nv=0.\end{equation}
Recall that $rad\langle~,~\rangle\subseteq Ann_{V_{(1)}}(V_{(0)})$. When $t=0$, we have $$u_nv=u_0q=0.$$ When $t=1$, using the fact that $rad\langle~,~\rangle$ is an ideal of $V_{(1)}$, we can conclude that
$$u_nv=u_{m_1}s^1_{-m_{1}}q=s^1_{-m_{1}}u_{m_1}q+(u_0s^1)_{0}q=0.$$ 
Now, we assume that equation (\ref{induction}) holds for all $b\in rad\langle~,~\rangle$, and for all $w\in V$ that has length $k\leq p$. Let $\phi\in rad\langle~,~\rangle$ and $v=s^1_{-m_{1}}....s^{p+1}_{-m_{{p+1}}}\tau\in V_{(l)}$ where $s^{i}\in V_{(1)}, \tau\in V_{(0)}, m_{i}>0$. Observe that
\begin{eqnarray*}
\phi_lv&&=\phi_l(s^1_{-m_{1}}....s^{p+1}_{-m_{{p+1}}}\tau)\\
&&=s^1_{-m_1}\phi_ls^2_{-m_2}....s^{p+1}_{-m_{{p+1}}}\tau+(\phi_0s^1)_{l-m_1}s^2_{-m_2}....s^{t+1}_{-m_{{p+1}}}\tau\\
&&\ \ \ \ \ +l(\phi_1s^1)_{l-m_1-1}s^2_{-m_2}....s^{p+1}_{-m_{{p+1}}}\tau\\
&&=(\phi_0s^1)_{l-m_1}s^2_{-m_{2}}....s^{p+1}_{-m_{p+1}}\tau.\\
\end{eqnarray*}
Since $\phi_0s^1\in rad\langle~,~\rangle$, $s^2_{-m_{2}}....s^{p+1}_{-m_{p+1}}\tau\in V_{(l-m_1)}$ and the length of $s^2_{-m_{2}}....s^{p+1}_{-m_{p+1}}\tau$ equals $p$, by an induction hypothesis, we then have that 
$$\phi_lv=0.$$ Therefore, for $u\in rad\langle~,~\rangle$, for $v\in V_{(n)}$ such that $v=s^1_{-m_{1}}....s^t_{-m_{t}}q$ where $s^{i}\in V_{(1)}, q^i\in V_{(0)}, m_{i_j}>0$, we have $u_nv=0$. By Proposition \ref{spanV}, we can conclude further that for $u\in rad\langle~,~\rangle$, $v\in V_{(n)}$, 
\begin{equation}\label{induction2}
u_nv=0.\end{equation}

\vspace{0.2cm}

\noindent Next, we will show that $(rad\langle~,~\rangle)\cap V_{(0)}=\{0\}$. Let $g\in (rad\langle~,~\rangle)\cap V_{(0)}$. Then there exist $u^i\in rad\langle~,~\rangle$, $v^i\in V_{(n_i)}$, $i\in\{1,...,p\}$ such that $g=\sum_{i=1}^p u^i_{n_i}v^i$. By equation (\ref{induction2}), we can conclude immediately that $g=0$. 

\vspace{0.2cm}

\noindent By using the fact that $(rad\langle~,~\rangle)$ is a proper ideal of $V$, we can conclude immediately that $(rad\langle~,~\rangle)=\{0\}$ when $V$ is simple. 
\end{proof}

\vspace{0.2cm}

\noindent Next, we recall definitions of a Lie algebroid and its module. 
\begin{dfn} Let $A$ be a commutative associative algebra. A Lie $A$-algebroid is a Lie algebra $\mathfrak{g}$ equipped with an $A$-module structure and a module action on $A$ by derivation such that 
$$[u,av]=a[u,v]+(ua)v,\ \ a(ub)=(au)b$$
for all $u,v\in \mathfrak{g},a,b\in A$.

\vspace{0.2cm}

\noindent A module for a Lie $A$-algebroid $\mathfrak{g}$ is a vector space $W$ equipped with a $\mathfrak{g}$-module structure and an $A$-module structure such that 
$$u(aw)-a(uw)=(ua)w,~a(uw)=(au)w$$ for $a\in A$, $u\in\mathfrak{g}$, $w\in W$. 
\end{dfn}

\begin{lem}\label{algebroidmodule} The space $V_{(1)}/Ann_{V_{(1)}}(V_{(0)})$ is a Lie $V_{(0)}$-algebroid. Moreover, $V_{(0)}$ is a module for a Lie $V_{(0)}$-algebroid $V_{(1)}/Ann_{V_{(1)}}(V_{(0)})$. \end{lem}
\begin{proof} First, we will show that $V_{(1)}/Ann_{V_{(1)}}(V_{(0)})$ is a Lie $V_{(0)}$-algebroid. Since $D(V_{(0)})\subseteq Ann_{V_{(1)}}(V_{(0)})$, we can conclude immediately that $V_{(1)}/Ann_{V_{(1)}}(V_{(0)})$ is a Lie algebra. Clearly, $V_{(1)}/Ann_{V_{(1)}}(V_{(0)})$ acts on $V_{(0)}$ as derivation. Recall that for $\alpha,t\in V_{(0)}$, we have 
$$((D(\alpha))_{-1}t)_0a'=\sum_{i\geq 0}(D(\alpha))_{-1-i}t_ia'+t_{-1-i}(D\alpha)_ia'=0$$ 
for all $a'\in V_{(0)}$. Hence, $(D(\alpha))_{-1}t\in Ann_{V_{(1)}}(V_{(0)})$ for all $\alpha,t\in V_{(0)}$. Furthermore, for $a,a'\in V_{(0)}$, $u\in V_1$, we have
\begin{eqnarray*}
&&(a_{-1}a')_{-1}u\\
&&=\sum_{i\geq 0}a_{-1-i}a'_{-1+i}u+a'_{-2-i}a_iu\\
&&=a_{-1}a'_{-1}u+a_{-2}a'_0u+a'_{-2}a_0u\\
&&=a_{-1}a'_{-1}u+(D(a))_{-1}a'_0u+(D(a'))_{-1}a_0u\\
&&\equiv a_{-1}a'_{-1}u\mod Ann_{V_{(1)}}(V_{(0)})
\end{eqnarray*}
and $V_{(1)}/Ann_{V_{(1)}}(V_{(0)})$ is a module of the commutative associative algebra $V_{(0)}$. Since
\begin{eqnarray*}
(u_0a_{-1}v)&&=a_{-1}u_0v+(u_0a)_{-1}v\\
(a_{-1}u)_{0}a'&&=\sum_{i\geq 0}a_{-1-i}u_ia'+u_{-1-i}a_ia'=a_{-1}u_0a
\end{eqnarray*}
for all $a,a'\in V_{(0)}$, $u,v\in V_{(1)}$, we can conclude immediately that $V_{(1)}/Ann_{V_{(1)}}(V_{(0)})$ is a Lie $V_{(0)}$-algebroid.

\vspace{0.2cm}

\noindent Next, we will show that $V_{(0)}$ is a module for a Lie $V_{(0)}$-algebroid $V_{(1)}/Ann_{V_{(1)}}(V_{(0)})$. Clearly, $V_{(0)}$ is a module of Leibniz algebra $V_{(1)}$. By commutativity and associativity, we have 
\begin{eqnarray*}
u_0a_{-1}a'&&=a_{-1}u_0a'+(u_0a)_{-1}a'\\
(a_{-1}u)_0a'&&=a_{-1}u_0a'.
\end{eqnarray*}
for $a,a'\in V_{(0)}$, $u\in V_{(1)}$. Hence, $V_{(0)}$ is a module for a Lie $V_{(0)}$-algebroid $V_{(1)}/Ann_{V_{(1)}}(V_{(0)})$. 
\end{proof}

\begin{rmk} $V_{(1)}/Ann_{V_{(1)}}(V_{(0)})$ is a vertex $V_{(0)}$-algebroid.
\end{rmk}

\vspace{0.2cm}

\noindent Next Theorem is another key ingredient for proving Theorem \ref{Main Theorem nonsimple}.
\begin{thm}\label{Vsimple} If $V$ is simple then $V_{(0)}$ is a simple module for a Lie $V_{(0)}$-algebroid $V_{(1)}/Ann_{V_{(1)}}(V_{(0)})$. As a result, $V$ is non-simple when $V_{(0)}$ is not a simple module for a Lie $V_{(0)}$-algebroid $V_{(1)}/Ann_{V_{(1)}}(V_{(0)})$. 
\end{thm} 
\begin{proof}
Let $W$ be a submodule of $V_{(0)}$. Let $(W)$ be an ideal of $V$ generated by $W$. We set $$U=Span\{u_nv~|~u\in W,~v\in V,~n\in\mathbb{Z}\}.$$ Clearly, $W\subseteq U$ and $D(U)\subseteq U$. Since $W\subseteq (W)$ and $(W)$ is an ideal of $V$, we can conclude that $U\subseteq (W)$. Now, we will show that $U=(W)$. Let $a\in V_{(0)}$, $b\in V_{(1)}$, $v\in V$, $u\in W$, $m,n\in\mathbb{Z}$. 
Since $W$ is a submodule of $V_{(0)}$, $a_mu_nv=u_na_mv\in U$ and $b_mu_nv=u_nb_mv+(b_0u)_{m+n}v\in U$, and $V$ is generated by $V_{(0)}$ and $V_{(1)}$, we can conclude that 
$s_tu_nv\in U$ for all $s,v\in V$, $u\in W$, $t,n\in\mathbb{Z}$. Hence, $U$ is an ideal of $V$ that contains $W$. In fact, $U=(W)$.

\vspace{0.2cm}

\noindent Recall that if $v$ is of the form $s^1_{-m_{1}}....s^t_{-m_{t}}q$ where $s^{i}\in V_{(1)}, q^i\in V_{(0)}, m_{i_j}>0$, then we say that $v$ has {\em length} $t$. We will use an induction on this type of length to show that for $u\in W$, and for $v\in V$ such that $v=s^1_{-m_{1}}....s^t_{-m_{t}}q\in V_{(n)}$ where $s^{i}\in V_{(1)}, q^i\in V_{(0)}, m_{i_j}>0$, we have 
\begin{equation}\label{induction3}
u_{n-1}v\in W\end{equation}
When $t=0$, we have $u_{-1}q=q_{-1}u\in W$. When $t=1$, 
$$u_{m_1-1}s^1_{-m_1}q=(s^1_{-m_1}u_{m_1-1}+(u_0s^1)_{-1})q=-(s^1_0u)_{-1}q\in W.$$
Hence, equation (\ref{induction3}) holds when $t=0$, and $t=1$. Now, we assume that equation (\ref{induction3}) holds for all $\alpha\in W$, and for all $\beta\in V$ that has length $k\leq p$. Let $w\in W$ and $v=s^1_{-m_{1}}....s^{p+1}_{-m_{{p+1}}}\tau\in V_{(l)}$ where $s^{i}\in V_{(1)}, \tau\in V_{(0)}, m_{i}>0$. By induction hypothesis, we have
\begin{eqnarray*}
w_{l-1}v&&=w_{l-1}(s^1_{-m_{1}}....s^{p+1}_{-m_{{p+1}}}\tau)\\
&&=s^1_{-m_1}w_{l-1}s^2_{-m_2}....s^{p+1}_{-m_{{p+1}}}\tau+(w_0s^1)_{l-m_1-1}s^2_{-m_2}....s^{t+1}_{-m_{{p+1}}}\tau\\
&&=(w_0s^1)_{l-m_1-1}s^2_{-m_{2}}....s^{p+1}_{-m_{p+1}}\tau\in W.
\end{eqnarray*}
By Proposition \ref{spanV}, we can conclude further that for $w\in W$, $v\in V_{(l)}$, 
\begin{equation}\label{induction3}
w_{l-1}v\in W.\end{equation}
Hence $(W)\cap V_{(0)}=W$. If $V$ is simple then $V_{(0)}$ is an irreducible module for a Lie $V_{(0)}$-algebroid $V_{(1)}/Ann_{V_{(1)}}(V_{(0)})$. 
\end{proof}

\begin{rmk} We set $V_{(0)}D(V_{(0)})=Span\{a_{-1}D(a')~|~a,a'\in V_{(0)}\}$. By associativity and the fact that $\alpha_0b=-b_0\alpha$ for $\alpha\in V_{(0)}$, $b\in V_{(1)}$, we have 
\begin{eqnarray*}
&&(\alpha_{-1}\alpha')_{-1}a_{-1}Da'=\alpha_{-1}\alpha'_{-1}a_{-1}Da',\\
&&(a_{-1}D(a'))_0V_{(0)}=0\text{ and }a_{-1}D(a')=D(a')_{-1}a
\end{eqnarray*} for all $\alpha,\alpha', a,a'\in V_{(0)}$. By following the proof in Lemma \ref{algebroidmodule} and Theorem \ref{Vsimple}, one can show that 

\vspace{0.2cm}

\noindent (i) $V_{(1)}/{V_{(0)}}D(V_{(0)})$ is a Lie $V_{(0)}$-algebroid. Moreover, $V_{(0)}$ is a module for a Lie $V_{(0)}$-algebroid $V_{(1)}/{V_{(0)}}D(V_{(0)})$. 

\vspace{0.2cm}

\noindent (ii) If $V$ is simple then $V_{(0)}$ is a simple module for a Lie $V_{(0)}$-algebroid $V_{(1)}/{V_{(0)}}D(V_{(0)})$. 

\vspace{0.2cm}

\noindent Note that these statements first appeared in \cite{LiY} for the case when $V$ is a vertex algebra associated with a vertex algebroid. 
\end{rmk}

\vspace{0.2cm}

\noindent We summarize the above discussion in the following Proposition.

\begin{prop}\label{nonsimpleV2} If $V$ is simple then $V_{(0)}$ is a simple module for a Lie $V_{(0)}$-algebroid $V_{(1)}/{V_{(0)}}D(V_{(0)})$. Hence, $V$ is non-simple when $V_{(0)}$ is not a simple module for a Lie $V_{(0)}$-algebroid $V_{(1)}/{V_{(0)}}D(V_{(0)})$.  
\end{prop}

\vspace{0.2cm}

\noindent We easily obtain Theorem \ref{Main Theorem nonsimple} by using Proposition \ref{Vindecomposable}, Theorem \ref{rad}, Theorem \ref{Vsimple} and Proposition \ref{nonsimpleV2}:

\vspace{0.1cm}

\noindent {\bf Theorem \ref{Main Theorem nonsimple} } Let $V=\oplus_{n=0}^{\infty}V_{(n)}$ be a $\mathbb{N}$-graded vertex algebra that satisfies the following properties:

\vspace{0.2cm}

\noindent (a) $2\leq \dim V_{(0)}<\infty$, $1\leq dim V_{(1)}<\infty$, $V$ is generated by $V_{(0)}$ and $V_{(1)}$;

\vspace{0.2cm}

\noindent (b) $V_{(0)}$ is a local algebra.

\vspace{0.2cm}
\noindent Assume that one of the following statements hold.

\vspace{0.2cm} 

\noindent (i) An ideal generated by $rad\langle~,~\rangle$ is not zero;

\vspace{0.2cm}

\noindent (ii) $V_{(0)}$ is not a simple module for a Lie $V_{(0)}$-algebroid $V_{(1)}/Ann_{V_{(1)}}(V_{(0)})$;

\vspace{0.2cm}

\noindent (iii) $V_{(0)}$ is not a simple module for a Lie $V_{(0)}$-algebroid $V_{(1)}/{V_{(0)}}D(V_{(0)})$. 

\vspace{0.2cm}

\noindent Then $V$ is an indecomposable non-simple vertex algebra. 

\vspace{0.2cm}

\noindent Now, Theorem \ref{MainTheorem} follows immediately from Theorem \ref{Main Theorem nonsimple}, Theorem \ref{Bsimple}, Remark \ref{BsimpleVnotsimple}, Theorem \ref{Bsemisimple}, and Corollary \ref{BsemisimpleVnotsimple}:

\vspace{0.2cm}

\noindent {\bf Theorem \ref{MainTheorem}.} Let $V=\oplus_{n=0}^{\infty}V_{(n)}$ be a $\mathbb{N}$-graded vertex algebra that satisfies the following properties

\vspace{0.2cm}

\noindent (a) $2\leq \dim V_{(0)}<\infty$. $1\leq dim V_{(1)}<\infty$, $V$ is generated by $V_{(0)}$ and $V_{(1)}$;

\vspace{0.2cm}

\noindent (b) $V_{(0)}$ is not a trivial module of a Leibniz algebra $V_{(1)}$, $u_{0}u\neq 0$ for some $u\in V_{(1)}$;

\vspace{0.2cm}

\noindent (c) the Levi Factor of $V_{(1)}$ equals $Span\{e,f,h\}$, $e_0f=h$, $h_0e=2e$, $h_0f=-2f$ and $e_1f=k{\bf 1}$. Here, $k\in\mathbb{C}\backslash \{0\}$.

\vspace{0.2cm}

\noindent Assume that one of the following statements hold.

\vspace{0.2cm} 

\noindent (i) $V_{(1)}$ is a simple Leibniz algebra;

\vspace{0.2cm}

\noindent (ii) $V_{(1)}$ is a semisimple Leibniz algebra and $Ker(D)\cap V_{(0)}=\{a\in V_{(0)}~|~b_0a=0\text{ for all } b\in V_{(1)}\}$. Here, $D$ is a linear operator on $V$ such that $D(v)=v_{-2}{\bf 1}$ for $v\in V$.
\vspace{0.2cm}

\noindent Then $V$ is an indecomposable non-simple vertex algebra.



\begin{thebibliography}{FKRW}
\bibitem[Ba1]{Ba} Barnes, D., Some theorems on Leibniz algebras, Comm. Algebra, \textbf{39(7)} (2011), 2463-2472.
\bibitem[Ba2]{Ba2} Barnes, D., On Levi's Theorem for Leibniz algebras, Bull. Aust. Math. Soc., \textbf{86(2)} (2012), 184-185.
\bibitem[Ba3]{Ba3} Barnes, D., On Engel's Theorem for Leibniz algebras, Comm. in Algebra \textbf{40} (2012), 1388-1389.
\bibitem[Bo]{Bo1} Borcherds, R. E., Vertex algebras, Kac-Moody algebras, and the Monster, Proc. Natl. Acad. Sci. USA \textbf{83} (1986), 3068-3071.

\bibitem[Br1]{Br1} Bressler, P., Vertex algebroids I, arXiv: math.AG/0202185.
\bibitem[Br2]{Br2} Bressler, P., Vertex algebroids II, arXiv: math.AG/0304115.

\bibitem[DMS]{DMS} Demir, I., Misra, K.C., and Stitzinger, E. On Some Structure of Leibniz Algebras, Contemporary Math. \textbf{623} (2014), 41-54.


\bibitem[DoLiM1]{DLM} Dong, C., Li, H.-S., and Mason, G., Modular-Invariance of Trace Functions in Orbifold Theory and Generalized Moonshine, Commun. Math. Phys. \textbf{ 214}, 1-56 (2000).

\bibitem[DoLiM2]{DLM2} Dong, C., Li, H.-S., and Mason, G., Vertex Lie algebra, vertex Poisson algebras and vertex algebras, in: Recent Developments in Infinite-Dimensional Lie Algebras and Conformal Field Theory, Proceedings of an International Conference at University of Virginia, May 2000, Contemporary Math. \textbf{297} (2002), 69-96.

\bibitem[DoM1]{DoM} Dong, C., and Mason, G., Local and Semilocal Vertex Operator Algebras, J. Alg., \textbf{280} (2004), 350-366. 

\bibitem[DoM2]{DoM2} Dong, C., and Mason, G., Shifted vertex operator algebras. Math. Proc. Cambridge Philos. Soc. \textbf{141} (2006), no. 1, 67-80.
\bibitem[FM]{FM} Fialowski, A., and Mihalka, E. Z. Representations of Leibniz Algebras, Algebras and Representation Theory \textbf{18} (2015), 477-490.

\bibitem[FrLMe]{FLM} Frenkel, I., Lepowsky, J., and Meurman, A., A natural repsentation of the Fisher-Griess Monster with the modular function $J$ as character, Proc. Acad. Sci. USA \textbf{81} (1984), 3256-3260.



\bibitem[GMS]{GMS} Gorbounov, V., Malikov, F., and Schechtman, V., Gerbes of chiral differential operators, II, Vertex algebroids, Invent. Math. \textbf{155} (2004), 605-680. 

\bibitem[LLi]{LLi} Lepowsky, L., and Li, H.-S., Introduction to Vertex Operator Algebras and Their Representations, Progr. Math., vol. \textbf{227}, Birkh\"{a}user, Boston (2003).
\bibitem[LiY]{LiY} Li, H.-S., and Yamskulna, G., On certain vertex algebras and their modules associated with vertex algebroids, J. Alg. \textbf{283} (2005) 367-398.
\bibitem[MS1]{MS1} Malikov, F., and Schechtman, V.,  Chiral de Rham complex, II. Differential Topology, Infinite -Dimensional Lie Algebras, and Applications, Amer. Math. Soc. Transl. Ser 2, vol \textbf{194},  Amer. Math. Soc., Providence, RI, 1999, 149-188.
\bibitem[MS2]{MS2} Malikov, F., and Schechtman, V., Chiral Poincare duality, Math. Res. Lett. \textbf{6} (1999), 533-546.
\bibitem[MSV]{MSV} Malikov, F., Schecthman, V., and Vaintrob, A., Chiral de Rham complex, Comm. Math. Phys. \textbf{204} (1999)439-473.
\bibitem[MaY]{MaY} Mason, G., and Yamskulna, G., On the structure of $\mathbb{N}$-graded vertex operator algebras. Developments and retrospectives in Lie theory, 247-274, Dev. Math.,\textbf{ 38}, Springer, Cham (2014). 
\bibitem[O]{O} Omirov, B. A., Conjugacy of Cartan subalgebras of complex finite-dimensional Leibniz algebras, J. Alg. \textbf{302} (2006), 887-896.
\end{thebibliography}
\end{document}